\documentclass[journal]{IEEEtran}

\usepackage[T1]{fontenc}    % 关键：支持重音字符（如ò）、特殊符号的正确编码
\usepackage[utf8]{inputenc} % 支持UTF-8编码输入（兼容中文/特殊字符）
\usepackage{amsthm}
\usepackage{cuted} % 必须加载
\usepackage{amsmath}
\newtheorem{theorem}{Theorem}
\newtheorem{lemma}{Lemma}
\usepackage{amsmath,amssymb,amsfonts}
\usepackage[greek,english]{babel}
\usepackage{upgreek}
\usepackage{setspace}
\newtheorem{assumption}{Assumption}
\newtheorem{remark}{Remark}
\usepackage{xcolor} 
\theoremstyle{definition}
\newtheorem{definition}{Definition}
\usepackage[usedefaultargs]{probsoln}
\usepackage{graphicx}
\usepackage{textcomp}
\usepackage{float}
\usepackage{xcolor}
\newtheorem{problem}{Problem}

\usepackage{algorithmic}
\usepackage{algorithm}

\usepackage{booktabs} % 用于专业表格排版的宏包
\usepackage{array}    % 增强表格功能
\usepackage{caption}
\usepackage{cite}
\usepackage{diagbox} % 必须添加这个宏包用于绘制斜线
\usepackage{booktabs}
\usepackage{cuted}
\usepackage{mathtools}

\allowdisplaybreaks[4]

\usepackage{enumitem}
\setlist[enumerate]{topsep=2pt,itemsep=0pt,parsep=0pt,partopsep=0pt}

\usepackage{multirow} % 多行合并
\usepackage{float}
\floatstyle{ruled}
\restylefloat{algorithm}

\newtheorem{example}{Example}

\makeatletter
\def\fs@ruled{%
  \let\@fs@cfont\bfseries
  \def\@fs@pre{\hrule height1pt \kern2pt}%
  \def\@fs@mid{\kern2pt\hrule height1pt\kern2pt}%
  \def\@fs@post{\kern2pt\hrule height1pt}%
  \let\@fs@iftopcapt\iftrue
}
\makeatother

\title{Noncooperative Game in Multi-controller System under Delayed and Asymmetric Information}

\author{Xin~Li,
        Qingyuan~Qi,
        Kemi~Ding
\thanks{Corresponding author: Kemi Ding.}
\thanks{
X. Li and K. Ding are with the School of Automation and Intelligent Manufacturing, Southern University of Science and Technology, Shenzhen, 518055, China (email: 12431354@mail.sustech.edu.cn, dingkm@sustech.edu.cn).}
\thanks{
Q. Qi is with the Qingdao Innovation and
Development Center, Harbin Engineering University, Qingdao, 266000, China (email: qingyuan.qi@hrbeu.edu.cn).}}

\begin{document}

\maketitle

\begin{abstract}
% \textcolor{blue}{
We address a noncooperative game problem in multi-controller system under delayed and asymmetric information structure.
Under these conditions, the classical separation principle fails as estimation and control design become strongly coupled, complicating the derivation of an explicit Nash equilibrium. To resolve this, we employ a common–private information decomposition approach, effectively decoupling control inputs and state estimation to obtain a closed-form Nash equilibrium. By applying a forward iterative method, we establish the convergence of the coupled Riccati and estimation error covariance recursions, yielding both the steady-state Kalman filter and the Nash equilibrium. Furthermore, we quantify the impact of asymmetric information, proving that a richer information set reduces the costs for the corresponding player. Finally, numerical examples are provided to demonstrate the effectiveness of the results.

\end{abstract}

% In this paper, we investigate the optimal control problem for a discrete-time system controlled by two controllers. In our formulation, two controllers can access a one-step noisy measurement from the two sensors, respectively. 
% Considering the unidirectional nature of information interaction, we assume that one controller can fuse measurement and action information received from another controller, 
% whereas the other controller has access only to its own information. 
%  We formulate this problem as a Noncooperative dynamic game under an asymmetric information structure, where each controller minimizes its own objective function. By confining the strategy space to linear policies, we derive the optimal control strategies by decomposing the global estimation information into common and private components. The Nash equilibrium is analytically characterized by a set of coupled Riccati equations, which admit a decoupled solution via backward recursion. Furthermore, through theoretical derivation and simulation verification, an important conclusion is obtained that more information contributes to cost reduction for the players in the game. The theoretical results are verified by numerical simulations.

% \textcolor{red}{
% Considering the sensors’ limited storage capacity, we assume that one sensor can fuse measurements received from another sensor, whereas the other sensor has access only to its own local measurements.}

\begin{IEEEkeywords}
% \textcolor{red}{
Cyber-physical systems, noncooperative game, multi-controller, delayed and asymmetric information structure, common-private information decomposition.
\end{IEEEkeywords}

\section{Introduction}
\label{sec:introduction}
Cyber-physical systems (CPSs) integrate computation, communication, and control for real-time interaction between physical processes and digital intelligence \cite{Baheti2011}. 
In many applications, CPSs consist of multiple distributed controllers connected through communication networks \cite{Lee2008}, where communication and sensing constraints may lead to delayed and asymmetric information, posing significant challenges for multi-controller systems.

\subsection{Related works}

Existing literature on multi-controller CPSs is largely restricted to collaborative scenarios, where information is shared to optimize a common performance criterion \cite{Yan2020, Murray2007}. These cooperative control frameworks are typically implemented through either centralized architectures, relying on a global coordinator \cite{Tsikalakis2011}, or distributed paradigms based on local interactions \cite{Sinopoli2003, Bamieh2002, DAndrea2003}.

However, in many practical control systems, multiple controllers often pursue distinct, potentially conflicting objectives rather than a common goal. Representative examples include modern power grids, where utilities and consumers strategically adjust generation and consumption to maximize their respective economic benefits \cite{Sheha2020, Wu2023}, as well as security-critical cyber-physical systems, where attackers and defenders interact in a fundamentally noncooperative manner \cite{Pirani2021}. In this context, noncooperative dynamic game theory provides a rigorous mathematical framework to model and analyze the strategic interactions among these self-interested decision-makers.
% A fundamental attribute governing such dynamic games is the information structure. Traditionally, a substantial body of literature has been developed under the assumption of a symmetric information structure, where all players share identical observation histories. Early studies primarily concentrated on the two-player setting \cite{Vorobeychik2004, Hefti2017, Tanimoto2007}, while subsequent efforts have successfully generalized the equilibrium analysis to multiplayer scenarios \cite{Bilo2021, Lv2018, Park2023}. Despite their theoretical elegance, symmetric information models often fall short in capturing the physical realities of modern decentralized systems, where sensors and communication channels are inherently distributed.
Within this game-theoretic framework, the information structure is a fundamental element. Most existing studies have focused on symmetric information structures, under which all players have access to identical observation histories. Early works mainly considered two-player games \cite{Vorobeychik2004, Hefti2017, Tanimoto2007}, and more recent work has generalized the analysis to the multiplayer
setting\cite{Bilo2021, Lv2018, Park2023}.  
Although the symmetric-information assumption is appropriate for some systems with fully shared observations, it may not be suitable for many decentralized CPSs, where controllers usually have access to different information due to distributed sensing and communication. Since asymmetric information among controllers commonly arises in practical scenarios, noncooperative games with asymmetric information have also received considerable attention.
% Driven by the prevalence of information asymmetry in networked systems,
% Motivated by the limitations of symmetric information models, stochastic games with asymmetric information have also been extensively studied.
However, the presence of asymmetric information among the controllers makes the explicit characterization of Nash equilibria significantly more challenging. To address this issue, Nayyar et al.~\cite{Nayyar2013} exploited the common information shared among the controllers to transform the original game into an equivalent symmetric-information game, from which the corresponding Markov perfect equilibria were obtained.
Ouyang et al. \cite{Ouyang2016} investigated dynamic games with asymmetric information under the PBE/CIB-PBE framework and proposed a sequential decomposition method based on public information. In subsequent work, this framework was further extended to dynamic games with hidden actions \cite{Vasal2016}.
It is stressed that such information asymmetry is equally prevalent in cooperative decentralized control.
In practical networked environments, communication and sensing delays may prevent controllers from accessing real-time measurements \cite{Vatanski2009}. 
Consequently, recent studies have considered cooperative decentralized control problems involving both one-step delays and asymmetric information \cite{Liang2023, nayyar2016optimal}.

Nevertheless, existing studies on dynamic games with asymmetric information provide primary theoretical frameworks, leaving the closed-form solutions for Nash equilibria largely unaddressed. Meanwhile, research on delayed information systems has been largely confined to cooperative control settings, with noncooperative strategic interactions remaining an open problem.
% To bridge this gap, this paper investigates a two-player noncooperative game featuring one-step delayed noisy measurements and an asymmetric information structure, and provides a closed-form Nash equilibrium. This formulation better reflects the realities of networked systems, where controllers may lack timely access to sensor measurements due to communication and processing delays, and the different abilities of different controllers naturally induce an asymmetric information structure among them. 
To bridge this gap, this paper investigates a two-player/controller noncooperative game featuring one-step delayed noisy measurements and asymmetric information. This formulation more accurately captures the operational characteristics of networked systems, in which communication and processing latencies may hinder timely access to sensor measurements, and the different capabilities of different controllers naturally induce an asymmetric information structure. Therefore, the article is of both theoretical and practical significance.

\subsection{Contribution}
In this paper, we consider a linear time-invariant system jointly regulated by two controllers under a delayed and asymmetric information structure. Specifically, controller 1 has access only to one-step delayed measurements, based on which a regularization objective is optimized self-interest. On the other hand, controller 2 possesses a richer information set, including its own measurements as well as the historical observations and control inputs of controller 1 transmitted through the communication network. We formulate this as a two-player noncooperative dynamic game under delayed and asymmetric information structure. 
Note that the considered information structure results in the classical separation principle failing, and the state estimation and control design are intrinsically coupled, which brings difficulties in obtaining the closed-form Nash equilibrium. To address these challenges, we employ a common-private information decomposition approach. Our method allows state estimation and control inputs to be effectively decoupled, thereby deriving the closed-form Nash equilibrium. 

% Note that the considered information structure results in the classical separation principle failing, and the state estimation and control design are intrinsically coupled, which brings difficulties in obtaining the closed-form Nash equilibrium. To address these challenges, we employ a common-private information decomposition approach. Our method allows state estimation and control inputs to be effectively decoupled, thereby deriving the closed-form Nash equilibrium. Due to
% the space limitation, the detailed proof of this article is provided in \cite{}.
% In addition, a forward iterative analysis of the coupled Riccati equations confirms the convergence to steady-state Nash equilibrium.

Compared with previous works, the main contributions of this article are summarized as follows.

% \textcolor{blue}{
(1)~A two-player noncooperative dynamic game with delayed and asymmetric information is formulated. Compared with existing game-theoretic models, the proposed formulation simultaneously accounts for delayed measurements and asymmetric information, thereby extending the analysis of noncooperative dynamic games to a nonclassical information structure.
% }

% \textcolor{blue}{
(2)~A common-private information decomposition is developed to handle the considered information structure. Based on this decomposition and a best-response analysis, a closed-form Nash equilibrium within the class of linear state-feedback strategies is derived and characterized by a set of coupled Riccati equations. Moreover, under standard stabilizability and detectability conditions, the corresponding steady-state Nash equilibrium is established.
% }

% \textcolor{blue}{
(3)~The impact of information on the cost is analyzed. By combining the monotonicity of the estimation error covariance with a backward iteration analysis of the Riccati matrices, we prove that richer information helps reduce the cost of the corresponding player. Numerical results verify the theoretical analysis.
% }

% \textcolor{red}{
The remainder of this article is organized as follows. In Section \ref{sec:model}, we formulate the problem of the considered noncooperative game. In Section \ref{mr}, we first employ a common-private information decomposition approach to find the Nash equilibrium and optimal cost functions, and analyze the relationship between cost and information structure. Furthermore, the convergence of the Nash equilibrium is provided. In Section \ref{ne}, a numerical example is provided to verify the theoretical results. Section \ref{cc} concludes the paper.

\textbf{Notation.} A n-dimensional Euclidean space is represented by $\mathbb{R}^n$, the set of $m \times n$ real matrices is denoted by $\mathbb{R}^{m \times n}$.
\(I_n\) is the identity matrix of dimension \(n\). For a matrix $A$, its transpose is $A^\top$, and its induced 2-norm by $\|A\|$. Given a square matrix $B$, $B \succ 0$ ($B \succeq 0$) means that $B$ is positive definite (positive semidefinite), and $\text{Tr}(B)$ is the trace of matrix $B$. For matrices $A$ and $B$, $\operatorname{diag}(A, B)$ denotes a (block) diagonal matrix with matrices $A, B$ on the diagonal. The $\sigma$-algebra generated by the random variable/vector $X$ is represented by $\sigma(X)$. $\mathbb{E}[\cdot]$ and $\operatorname{Cov}(\cdot)$ denote the mathematical expectation and covariance, respectively.
% Moreover, $\mathbb{E}(x)$ is the mathematical expectation of the random variable $x$. 
$\mathbb{E}[\text{•}|\mathcal{F}_k]$ means that the conditional expectation with respect to $\mathcal{F}_k$. A finite sequence (or set) $\{y_0, y_1, y_2, \cdots, y_N\}$ is denoted by $\{y_k\}_{k=0}^N$. 
% \textcolor{blue}{
Throughout this paper, the superscripts 1 and 2 are used to distinguish the two controllers, sensors, estimators, and devices.

\section{Problem Formulation}\label{sec:model}
\subsection{System Model}
% \begin{figure}[htbp]
%     \centering
%     % 也可以使用 \linewidth，它在单栏中等同于 \textwidth
%   \includegraphics[width=1.0\linewidth]{figures/ffggg.pdf}
%     \caption{System model under the one-step delayed measurements and asymmetric information structure}
%     \label{fig1}
% \end{figure}
\begin{figure}[H]
    \centering
    % 也可以使用 \linewidth，它在单栏中等同于 \textwidth
  \includegraphics[width=\linewidth]{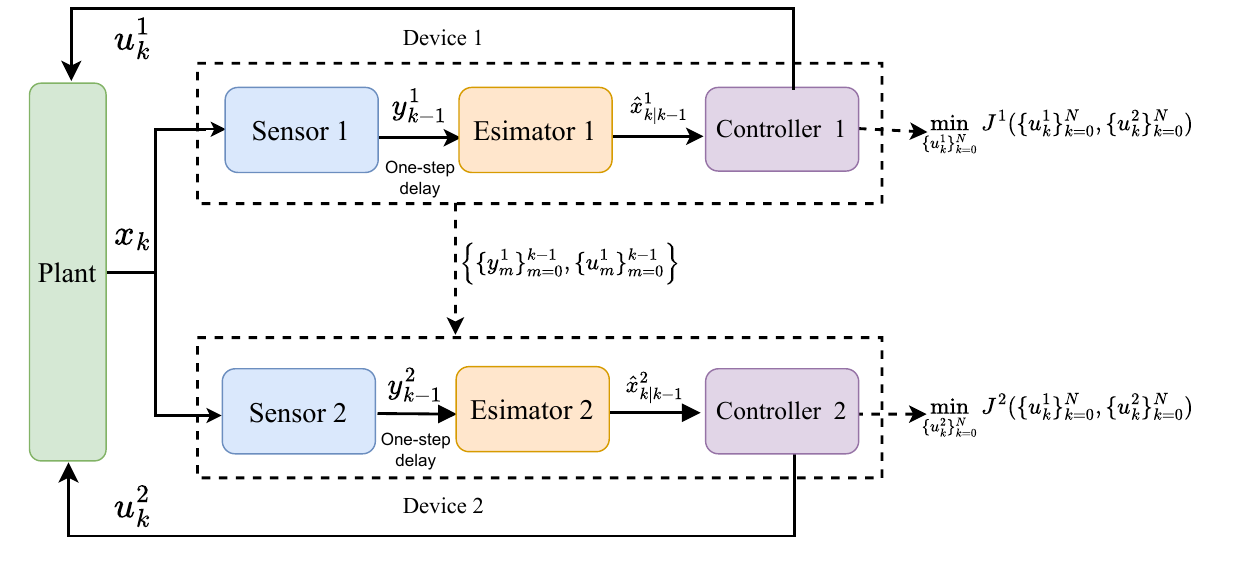}
    \caption{System model under the delayed and asymmetric information structure}
    \label{fig1}
\end{figure}

% \textcolor{blue}{
As shown in Figure~\ref{fig1}, we consider a linear time-invariant (LTI) system driven by two devices (i.e., Device 1 and Device 2), each embedded with a sensor, an estimator, and a controller. The sensor monitors the plant, the estimator processes the observations, and the controller computes and generates the control signals. 
The two devices operate under an asymmetric information structure induced by a unidirectional information flow from Device~1 to Device~2.
Specifically, Device~2 can access past observations and control inputs of Device 1, whereas Device~1 observes only its own local information. Considering communication delays between sensors and estimators, both estimators receive one-step delayed sensor measurements.
Consequently, each controller optimizes its own objective based on one-step delayed state estimation. A practical example of this setting will be introduced later in Example 1.
% }

In this paper, the dynamics of the discrete-time LTI system is given by
\begin{equation}\label{s1}
\left\{
\begin{aligned}
x_{k+1} &= Ax_k + B_1u_k^1 + B_2u_k^2 + w_k,\\
y_k^1 &= C_1x_k + v_k^1,\\
y_k^2 &= C_2x_k + v_k^2,
\end{aligned}
\right.
\end{equation}
where $x_k \in \mathbb{R}^n$ is the system state, the control inputs are denoted by $u_k^1 \in \mathbb{R}^{m_1}$, $u_k^2 \in \mathbb{R}^{m_2}$, respectively. The measurements of sensors 1 and 2 are given by $y_k^1 \in \mathbb{R}^{p_1}$ and $y_k^2 \in \mathbb{R}^{p_2}$, respectively. The system matrix is represented by $A \in \mathbb{R}^{n \times n}$ , the control input matrices are denoted by $B_1 \in \mathbb{R}^{m_1 \times n}, B_2\in \mathbb{R}^{m_2 \times n}$, and $C_1 \in \mathbb{R}^{p_1 \times n}$, $C_2 \in \mathbb{R}^{p_2 \times n}$ are the measurement matrices. 
Here, $\omega_k\in\mathbb{R}^n$ denotes the system noise, while $v_k^{1}\in\mathbb{R}^{p_1}$ and $v_k^{2}\in\mathbb{R}^{p_2}$ denote the measurement noises of the two sensors, respectively.
The initial state $x_0$ and the noises $w_k$, $v_k^{1}$, and $v_k^{2}$ are mutually independent Gaussian random variables with means $\mu$, $0$, $0$, and $0$, and covariances $\sigma$, $Q_w$, $Q_{v^{1}}$, and $Q_{v^{2}}$, respectively.

Consider the finite-horizon problem over $k \in \{0, 1, \dots, N\}$, each controller acts as a regulator for system~\eqref{s1} and selects its control input to minimize its own cost function, given by
\begin{small}
\begin{equation}\label{cs}
\begin{aligned}
&J^i(\{u_k^1\}_{k=0}^N, \{u_k^2\}_{k=0}^N)=  \mathbb{E}\Big[\sum_{k=0}^{N} \Big( x_k^\top Q_i x_k + (u_k^1)^\top S_i u_k^1\\
&+(u_k^2)^\top R_i u_k^2 \Big)\Big]+ x_{N+1}^{\top}P_{N+1}^ix_{N+1}, \quad i=1, 2,
\end{aligned}
\end{equation}
\end{small}
% \textcolor{blue}{
where $Q_i$, $S_i$, $R_i$, and $P_{N+1}^i, i=1, 2$ are positive semi-definite weight matrices, such that the cost functions of controllers are coupled with each other through their control input and the state of the system. This article focuses on designing controllers for each subsystem to minimize their individual cost functions in \eqref{cs}.
% }

\begin{example}
In practice, the power network adopts a two-level hierarchical control architecture \cite{Scattolini2009}, consisting of an upper-level coordinator and multiple lower-level model predictive control (MPC) regulators. The coordinator generates reference trajectories and boundary constraints based on the global objectives, while the local MPC regulators perform closed-loop control to optimize dynamic response and tracking performance. The upper level sends reference values to the lower level, but the lower level does not pass feedback information to the upper level. This architecture is characterized by a unidirectional information flow, which induces an asymmetric information structure. Consequently, this problem can be formulated as an optimization problem with distinct objectives under an asymmetric information structure.
\end{example}
% \textcolor{blue}{The explanation of this example is not clear to show the connection with your setting. Any other existing work consider such unidirectional information flow between two devices?}
% \end{example}
\vspace{-1.5em}
\subsection{Asymmetric Information Structure}
Recall that, as shown in Figure \ref{fig1}, controller~2 has access to its own one-step delayed signal $\{y_k^2\}_{m=0}^{k-1}$, $\{u_k^2\}_{m=0}^{k-1}$, as well as the one-step delayed signal $\{y_k^1\}_{m=0}^{k-1}$ and $\{u_k^1\}_{m=0}^{k-1}$ of controller~1. On the other hand, controller~1 receives only its own local information $\{y_k^1\}_{m=0}^{k-1}$, $\{u_k^1\}_{m=0}^{k-1}$. 
% The main challenge for solving Problem 1 \textcolor{blue}{(where is problem 1??)} is the asymmetry of the information structure, making the design of optimal control strategies difficult.
Accordingly, the information sets available to controllers~1 and~2 can be characterized as follows.
\begin{align}\label{infor1}
\mathcal{F}_k^1&=\sigma\Big\{ \{y_m^1\}_{m=0}^{k-1}, \{u_m^1\}_{m=0}^{k-1}\Big\},
\notag\\
\mathcal{F}_k^2&=\sigma\Big\{ \{y_m^1\}_{m=0}^{k-1}, \{u_m^1\}_{m=0}^{k-1}, \{y_m^2\}_{m=0}^{k-1}, \{u_m^2\}_{m=0}^{k-1}\Big\},
\end{align}
where $u_k^1$ is $\mathcal{F}_k^1$--measurable and $u_k^2$ is $\mathcal{F}_k^2$--measurable, respectively. Obviously, we have $\mathcal{F}_k^1 \subset \mathcal{F}_k^2$.

\subsection{The Game-theoretic Formulation}
The objective of this work is to solve the coupled optimization problem described in \eqref{jj2} within a networked environment, which induces one-step delayed measurements and an asymmetric information structure between the two devices. Here, historical information can only be transmitted from Device 1 to Device 2. It is worth noting that the two controllers have distinct objectives and interact with each other, which can be modeled as a noncooperative game under delayed and asymmetric information structure.

\textbf{Noncooperative dynamic game.}
The noncooperative game under asymmetric information induced by \eqref{s1}--\eqref{infor1} is defined as the tuple
\(
\mathcal G \triangleq \langle \mathcal I,\ \mathcal S,\ \{\mathcal F_k^i\}_{i\in\mathcal I},\ \{\mathcal U_i\}_{i\in\mathcal I},\ \{J^i\}_{i\in\mathcal I} \rangle,
\)
where:
\begin{itemize}
\item \textbf{Players:} $\mathcal I \triangleq \{\text{Controller~1},\,\text{Controller~2}\}$. It is assumed that Controllers 1 and 2 are rational players; that is, each selects an action that best serves its own objective from its available action set.
\item \textbf{State space:} $\mathcal S \triangleq \mathbb R^n$, with state $x_k\in\mathcal S$ evolving according to \eqref{s1}.
\item \textbf{Information structure:} At time $k$, player $i \in \mathcal{I}$ chooses $u_k^i$ based on its available information set $\mathcal F_{k}^i$, respectively.
\item \textbf{Action space:}
To characterize the set of admissible \cite{T1998} action for player $i$ over the finite horizon, we define the admissible action set as
\begin{small}
\begin{align*}
\mathcal{U}_i \triangleq \Big\{ \{u_k^i\}_{k=0}^{N} \mid \ & u_k^i \in \mathbb{R}^{m_i}, u_k^{i} \text{ is } \mathcal{F}_k^i\text{-measurable}, \\
& \sum_{k=0}^N \mathbb{E} \left[ \|u_k^i\|^2\right] < +\infty \Big\}.
\end{align*}
\end{small}
It is worth noting that $\mathcal{U}_i$ is non-empty, closed and convex subset of $\mathbb{R}^{m_i}$, and $\{u_{k}^i\}_{k=0}^{N} \in \mathcal{U}_i$ is called admissible action sequence.
The strategy at time $k$ is a mapping $
\mu_k^i:\mathcal F_k^i \to \mathcal U_i$,
and action $u_k^i=\mu_k^i(\mathcal {F}_k^i)$.
A sequence $\pi^i=\{\mu_k^i\}_{k=0}^N$ is called an admissible policy for player $i$.

\item \textbf{Costs:} 
For each player \(i \in \{1,2\}\), given the joint action sequences
\(\{u_k^1\}_{k=0}^N \in \mathcal{U}_1\) and \(\{u_k^2\}_{k=0}^N \in \mathcal{U}_2\),
the finite-horizon quadratic cost \(J^i\) is given by \eqref{cs}.
% For player $i=1, 2$, given an action sequence $\{u_{k}^i\}_{k=0}^{N} \in \mathcal{U}_i$, the finite-horizon quadratic cost $J^i$ follows \eqref{cs}.
\end{itemize}

Moreover, Nash equilibrium is treated as the common solution of the noncooperative game, and its definition for game $\mathcal{G}$ is given as follows.

\begin{definition}\label{def1}
A pair \((\{u_k^{1,*}\}_{k=0}^N,\{u_k^{2,*}\}_{k=0}^N)\) is called a Nash equilibrium of game \(\mathcal G\) if, for any admissible action sequences
\(\{u_k^{1}\}_{k=0}^N \in \mathcal U_1\) and
\(\{u_k^{2}\}_{k=0}^N \in \mathcal U_2\), the following inequalities hold:
\begin{align}
J^1\bigl(\{u_k^{1,*}\}_{k=0}^N,\{u_k^{2,*}\}_{k=0}^N\bigr)
&\le
J^1\bigl(\{u_k^{1}\}_{k=0}^N,\{u_k^{2,*}\}_{k=0}^N\bigr), \notag\\
J^2\bigl(\{u_k^{1,*}\}_{k=0}^N,\{u_k^{2,*}\}_{k=0}^N\bigr)
&\le
J^2\bigl(\{u_k^{1,*}\}_{k=0}^N,\{u_k^{2}\}_{k=0}^N\bigr).
\end{align}
\end{definition}

We formulate the problem of finding the Nash equilibrium of the noncooperative dynamic game as follows.
\begin{problem}\label{prob1}
% Consider the noncooperative dynamic game \(\mathcal G\). 
Under the asymmetric information structure \eqref{infor1}, the problem is to determine a Nash equilibrium of the game \(\mathcal G\). For each player \(i\in\{1,2\}\), given the strategy of the other player, the goal is to find an admissible control sequence \(\{u_k^i\}_{k=0}^N\) that minimizes its own cost functional. Accordingly, player \(i\) solves
\begin{small}
\begin{equation}\label{eq:opt_i}
\begin{aligned}
\min_{\{u_k^{i}\}_{k=0}^N}\quad & J^i\big(\{u_k^{1}\}_{k=0}^N,\{u_k^{2}\}_{k=0}^N\big)\\
\text{s.t.}\quad
&u_k^{i}\in \mathcal U_i,\quad \forall k \in \{0, \dots, N\}.
\end{aligned}
\end{equation}
\end{small}
\end{problem}

\begin{remark}
% \textcolor{blue}{
To the best of our knowledge, Problem \ref{prob1} has not been previously addressed in the literature. Existing works have focused on complete history sharing among controllers \cite {Ouyang2016, Nayyar2013, Vasal2016}, while this article addresses a more general noncooperative game under one-step delayed measurement and asymmetric information. The core challenge stems from the nonclassical information pattern, which induces a coupling between state estimation and controller design, thereby complicating the derivation of a closed-form Nash equilibrium.
\end{remark}

\section {Main Results}\label{mr}
% \textcolor{blue}{
This section is devoted to the explicit characterization of the Nash equilibrium under the one-step delayed measurement and the asymmetric information structure. Then we establish the monotonicity of the optimal cost of the corresponding player with respect to the information structure. Furthermore, both the steady-state Kalman filtering and the Nash equilibrium are obtained via a forward iterative scheme.

\subsection{Common-private Information Decomposition}

To overcome the challenge caused by asymmetric information structure, we exploit the nested information pattern defined in \eqref{infor1}, and employ a common-private information decomposition approach to analyze the output of estimators and control inputs $u^1_k$, $u^2_k$ as follows.

First, let \(\hat x_{k|k-1}^1=\mathbb E[x_k\mid \mathcal F_k^1]\) and \(\hat x_{k|k-1}^2=\mathbb E[x_k\mid \mathcal F_k^2]\) denote the one-step state estimate available to controllers~1 and~2, respectively. Since \(\mathcal F_k^1 \subset \mathcal F_k^2\), controller~2 has access to a richer information set than controller~1. Then, the state estimate \(\hat x_{k|k-1}^2\) can be decomposed into a common part and a private part, where the common part is \(\hat{x}_{k|k-1}^1\) and the private part is defined as \(
d_k \triangleq \hat{x}_{k|k-1}^2-\hat{x}_{k|k-1}^1.\) By the orthogonality principle of conditional expectation \cite{Durrett2019}, it holds that \(\mathbb{E}\![d_k \mid \mathcal F_k^1]=0\) and \(\mathbb{E}\!\left[\hat{x}_{k|k-1}^{\,1} d_k^\top\right]=0\).
Moreover, we define $e_k^1 \triangleq x_k-\hat{x}_{k|k-1}^1$ and $e_k^1 \triangleq x_k--\hat{x}_{k|k-1}^2$, then we have $e_k^1 = e_k^2 + d_k$. By using the orthogonality principle of conditional expectation, we further have the following conclusions induced by the common-private information decomposition.
\begin{lemma}\label{lem1}
For all $k \in \{0, \dots, N\}$ and $e_k^2=x_k-\hat{x}_{k|k-1}^{\,2}$, the following orthogonality relations hold:
\begin{small}
\begin{equation}
\mathbb{E}\!\left[\hat{x}_{k|k-1}^{\,1} (e_k^2)^\top\right]=0,\quad
\mathbb{E}\!\left[d_k (e_k^2)^\top\right]=0.
\end{equation}
\end{small}
\end{lemma}

Next, based on the measurability of $u^2_k$ with respect to \(\mathcal F_k^1\), the following definitions are given:
\begin{small}
\begin{equation}\label{f1}
\hat{u}_k^2 \triangleq \mathbb{E}\!\left[u_k^2 \mid \mathcal{F}_{k}^1\right], 
\qquad 
\tilde{u}_k^2 \triangleq u_k^2-\hat{u}_k^2.
\end{equation}
\end{small}
By computation, $\hat{u}_k^2$ and  $\tilde{u}_k^2$ satisfy
\begin{small}
\begin{equation}\label{f2}
\begin{aligned}
&\mathbb{E}[\tilde{u}_k^2]=0, \qquad
\mathbb{E}\!\left[\tilde{u}_k^2 \mid \mathcal{F}_{k}^1\right]=0,\\
&\mathbb{E}\!\left[\hat{u}_k^2 \mid \mathcal{F}_{k}^2\right]=\hat{u}_k^2, \qquad
\mathbb{E}\!\left[\tilde{u}_k^2 \mid \mathcal{F}_{k}^2\right]=\tilde{u}_k^2.
\end{aligned}
\end{equation}
\end{small}
As a consequence, $u_k^2$ is decomposed into $\hat{u}_k^2$ and $\tilde{u}_k^2$. Note that $u_k^1$ and $\hat{u}_k^2$ have same measurability respect to $\mathcal F_k^1$, then we can define
$\hat{u}_k=\begin{bmatrix}
u_k^1 \\
\hat{u}_k^2
\end{bmatrix}$.
% \textcolor{red}{
Furthermore, we assume that the players employ strategies in the following linear state-feedback form:
\begin{small}
\begin{align}\label{u2}
\hat{u}_k&=K_k^1\,\hat{x}_{k|k-1}^1, \quad \tilde{u}_k^2=K_k^2\big(\hat{x}_{k|k-1}^2-\hat{x}_{k|k-1}^1\big),
% u_k^2&=K_k^1\,\hat{x}_{k|k-1}^1+K_k^2\big(\hat{x}_{k|k-1}^2-\hat{x}_{k|k-1}^1\big)
\end{align}
\end{small}
where \(K_k^1\) and \(K_k^2\) denote the state-feedback gains to be determined.

In addition, we define \(B \triangleq [\,B_1,\; B_2\,]\), \(
\Gamma_1 \triangleq 
\begin{bmatrix}
S_1 & 0\\
0 & R_1
\end{bmatrix}\), and \(\Gamma_2 \triangleq 
\begin{bmatrix}
S_2 & 0\\
0 & R_2
\end{bmatrix}
\). Since \(S_1\), \(S_2\), \(R_1\) and \(R_2\) defined in \eqref{cs} are all positive semidefinite, hence the block-diagonal matrices \(\Gamma_1, \Gamma_2\) are both positive semidefinite. Consequently, on the basis of \eqref{f1}-\eqref{f2}, the system dynamics \eqref{s1} can be reformulated into an equivalent form
\begin{small}
\begin{align}\label{rs1}
x_{k+1} = Ax_k + B\hat{u}_k + B_2\tilde{u}_k^2+w_k.
\end{align}
\end{small}
For $i=1, 2$, the associated cost functions \eqref{cs} can equivalently be represented as
\begin{small}
\begin{equation}\label{jj2}
\begin{aligned}
&J^i(\{\hat{u}_k\}_{k=0}^N,\{\tilde{u}_k^{2}\}_{k=0}^N) = \mathbb{E}\left[ x_{N+1}^{\top} P_{N+1}^i x_{N+1} \right]\\
&\qquad+ \sum_{k=0}^N \mathbb{E}\Big[ x_k^{\top} Q_i x_k + \hat{u}_k^{\top} \Gamma_i \hat{u}_k + (\tilde{u}_k^2)^{\top} R_{i} \tilde{u}_k^2\Big].
\end{aligned}
\end{equation}
\end{small}

\begin{remark}
The proposed common-private decomposition approach helps address the coupling difficulty under one-step delayed measurement and asymmetric information. Through the simultaneous decomposition of state estimates and control inputs, the linear strategy form can be established, which facilitates the analytical derivation of the optimal state estimation and Nash equilibrium.
\end{remark}
\vspace{-1em}

\subsection{The state estimate via Kalman filter}
% In this subsection, we investigate the optimal state estimations for the estimators 1 and 2 with different information by employing the results of common-private information.
In this subsection, we derive the optimal state estimates for estimators 1 and 2 under different information structures.

In line with standard assumptions in linear estimation and control \cite{Kalman1960, Kalman1963, Chen1984}, we impose the following assumption on system~\eqref{s1}.
% Following standard assumptions in linear estimation and control \cite{Kalman1960, Kalman1963, Chen1984}, we impose the following standing assumption on system~\eqref{s1}.

\begin{assumption}\label{ass1}
For \(i=1, 2\), the pair \((A, B_i)\) is controllable and the pair \((A, C_i)\) is observable.
\end{assumption}
Since estimator 2 has access to the observations 
$\{y_m^1\}_{m=0}^{k-1}$ and $\{y_m^2\}_{m=0}^{k-1}$, 
its measurement equation at time $k$ can be written as
\begin{small}
\begin{equation}\label{yy1}
    y_k = Cx_k + v_k,
\end{equation}
\end{small}
where $y_k=\begin{bmatrix}
y_k^{1} \\
y_k^{2}
\end{bmatrix}$, 
$v_k=
\begin{bmatrix} 
v_k^1\\
v_k^2
\end{bmatrix}$, 
and $C=
\begin{bmatrix}
C_{1}\\
C_{2}
\end{bmatrix}$.

In addition to Assumption~\ref{ass1}, we assume that $(A, C)$ is observable. 
Then, by standard Kalman filtering \cite{Anderson2005}, the optimal state estimation for the two estimators is characterized in the following lemma.

% Under the linear strategy class, we adopt $\hat{u}_k = K^1 \hat{x}_{k|k-1}^1$ and
% $\tilde{u}_k^2 = K^2\big(\hat{x}_{k|k-1}^2-\hat{x}_{k|k-1}^1\big)$.

\begin{lemma}\label{kf_lem}
Based on \eqref{f1}-\eqref{rs1}, the estimates of estimators 1 and 2 are represented as
% Using \eqref{u2}, \eqref{rs1} and \eqref{yy1}, the estimates of estimators 1 and 2 are represented as

\begin{equation}\label{kf}
\left\{
\begin{aligned}
\hat{x}_{k|k}^1&=\hat{x}_{k|k-1}^1+G_{k|k-1}^1(y_k^1-C_1\hat{x}_{k|k-1}^1),\\
\hat{x}_{k|k-1}^1&=A\hat{x}_{k-1|k-1}^1+B\hat{u}_{k-1},\\
\hat{x}_{k|k}^2&=\hat{x}_{k|k-1}^2+G_{k|k-1}^2(y_k-C\hat{x}_{k|k-1}^2),\\
\hat{x}_{k|k-1}^2&=A\hat{x}_{k-1|k-1}^2+B\hat{u}_{k-1}+B_2\tilde{u}_{k-1}^2,
\end{aligned}
\right.
\end{equation}
where the estimation error covariance matrices of estimators 1 and 2 are given by
\begin{small}
\begin{equation}\label{kf1}
\begin{aligned}
&\Sigma_{k|k}^1\hspace{-1mm}=\hspace{-1mm}(I\hspace{-1mm}-\hspace{-1mm}G_{k|k-1}^{1}C_1)\Sigma_{k|k-1}^1(I\hspace{-1mm}-\hspace{-1mm}G_{k|k-1}^{1}C_1)^{\top}\hspace{-1mm}+\hspace{-1mm}G_{k|k-1}^{1}Q_{v_1}(G_{k|k-1}^{1})^{\top}\hspace{-1mm},\\
&G_{k|k-1}^{1}\hspace{-1mm}= \hspace{-1mm}\Sigma_{k|k-1}^{1} C_1^{\top}\big(C_1 \Sigma_{k|k-1}^{1} C_1^{\top}\hspace{-1mm}+\hspace{-1mm}Q_{v_1}\big)^{-1},\\
&\Sigma_{k+1|k}^1
=(A(I-G_{k|k-1}^{1}C_1)\hspace{-0.5mm}+\hspace{-0.5mm}B_2K_k^2)\,(\Sigma_{k|k-1}^1\hspace{-0.5mm}-\hspace{-0.5mm}\Sigma_{k|k-1}^2)\,\\
&\times(A(I-G_{k|k-1}^{1}C_1)\hspace{-0.5mm}+\hspace{-0.5mm}B_2K_k^2)^{\top} \hspace{-0.5mm}+\hspace{-0.5mm}A(I-G_{k|k-1}^{1}C_1)\,\\
&\times\Sigma_{k|k-1}^2\,(I\hspace{-1mm}-\hspace{-1mm}G_{k|k-1}^{1}C_1)^{\top}A^{\top}
\hspace{-1mm}+\hspace{-1mm}A G_{k|k-1}^{1} Q_{v_1} (G_{k|k-1}^{1})^{\top}A^{\top}\hspace{-1mm}+\hspace{-1mm}Q_w,
\end{aligned}
\end{equation}
\end{small}
\begin{small}
\begin{equation}\label{kf2}
\begin{aligned}
&\Sigma_{k|k}^2\hspace{-1mm}=\hspace{-1mm}(I\hspace{-1mm}-\hspace{-1mm}G_{k|k-1}^2C)\Sigma_{k|k-1}^2(I\hspace{-1mm}-\hspace{-1mm}G_{k|k-1}^2C)^{\top}\hspace{-1mm}+\hspace{-1mm}G_{k|k-1}^2Q_v\left(G^{2}_{k|k-1}\right)^{\top},\\
&G_{k|k-1}^2 = \Sigma_{k|k-1}^2 C^{\top} (C \Sigma_{k|k-1}^2C^{\top}+Q_{v})^{-1},\\
&\Sigma_{k+1|k}^2=A\Sigma_{k|k}^2A^{\top}+Q_w,
\end{aligned}
\end{equation}
\end{small}
with initial value $x_{0|-1}^1=x_{0|-1}^2=\mu$, and $\Sigma_{0|-1}^1=\Sigma_{0|-1}^2=\sigma$. 
% and $\Sigma_{0|-1}^1=\sigma_1$, $\Sigma_{0|-1}^2=\sigma_2$.
\end{lemma}

\begin{proof}
Using \eqref{f1}, \eqref{f2}, and \eqref{rs1}, the optimal estimates of estimators 1 and 2 can be directly computed via the Kalman filter. Next, the estimation error covariance of estimator 1 satisfies
\begin{small}
\begin{equation}\label{eec1}
\begin{aligned}
\Sigma_{k+1|k}^1
\hspace{-0.5mm}=\hspace{-0.5mm}\Big[(A(I\hspace{-1mm}-\hspace{-1mm}G_{k|k-1}^{1}C_1)(x_k\hspace{-1mm}-\hspace{-1mm}\hat{x}_{k|k-1}^1)\hspace{-1mm}-\hspace{-1mm}AG_{k|k-1}^{1}v_k^1\hspace{-1mm}+\hspace{-1mm}B_2\tilde{u}_k^2\hspace{-1mm}+\hspace{-1mm}w_k\Big]\\
\times \Big[(A(I\hspace{-1mm}-\hspace{-1mm}G_{k|k-1}^{1}C_1)(x_k-\hat{x}_{k|k-1}^1)\hspace{-1mm}-\hspace{-1mm}AG_{k|k-1}^{1}v_k^1\hspace{-1mm}+\hspace{-1mm}B_2\tilde{u}_k^2+w_k\Big]^{\top}, 
\end{aligned}
\end{equation}
\end{small}
and substituting the linear state-feedback strategies \eqref{u2} into \eqref{eec1}, then
\begin{small}
\begin{equation}\label{eec2}
\begin{aligned}
\Sigma_{k+1|k}^1
&\hspace{-0.5mm}=\hspace{-0.5mm}\Big[(A(I\hspace{-1mm}-\hspace{-1mm}G_{k|k-1}^{1}C_1)(x_k\hspace{-1mm}-\hspace{-1mm}\hat{x}_{k|k-1}^1)\hspace{-1mm}-\hspace{-1mm}AG_{k|k-1}^{1}v_k^1\hspace{-1mm}\\
&+\hspace{-1mm}B_2(\hat{x}_{k|k-1}^2-\hat{x}_{k|k-1}^1)^{\top})\hspace{-1mm}+\hspace{-1mm}w_k\Big]\\
&\times \Big[(A(I\hspace{-1mm}-\hspace{-1mm}G_{k|k-1}^{1}C_1)(x_k-\hat{x}_{k|k-1}^1)\hspace{-1mm}\\
&-\hspace{-1mm}AG_{k|k-1}^{1}v_k^1\hspace{-1mm}+\hspace{-1mm}B_2(\hat{x}_{k|k-1}^2-\hat{x}_{k|k-1}^1)^{\top})+w_k\Big]^{\top}.
\end{aligned}
\end{equation}
\end{small}
% Using the orthogonality conclusion given in Lemma \ref{lem1}, we have
By utilizing the orthogonality conclusion in Lemma \ref{lem1}, we have
\begin{equation}\label{eec2}
\begin{aligned}
&\mathbb{E}[(x_k-\hat{x}_{k|k-1}^1)(\hat{x}_{k|k-1}^2-\hat{x}_{k|k-1}^1)^{\top})]\\
&=\mathbb{E}\{(x_k-\hat{x}_{k|k-1}^1)[(x_k-\hat{x}_{k|k-1}^1)-(x_k-\hat{x}_{k|k-1}^2)]\}\\
&=\Sigma_{k|k-1}^1-\mathbb{E}\{[(x_k-\hat{x}_{k|k-1}^2)+\hat{x}_{k|k-1}^2](x_k-\hat{x}_{k|k-1}^2)^{\top}\}\\
&=\Sigma_{k|k-1}^1-\Sigma_{k|k-1}^2,
\end{aligned}
\end{equation}
therefore, \eqref{eec1} becomes 
\begin{small}
\begin{equation}\label{eec3}
\begin{aligned}
&\Sigma_{k+1|k}^1=(A(I-G_{k|k-1}^{1}C_1)\hspace{-0.5mm}+\hspace{-0.5mm}B_2K_k^2)\,(\Sigma_{k|k-1}^1\hspace{-0.5mm}-\hspace{-0.5mm}\Sigma_{k|k-1}^2)\,\\
&\times(A(I-G_{k|k-1}^{1}C_1)\hspace{-0.5mm}+\hspace{-0.5mm}B_2K_k^2)^{\top} \hspace{-0.5mm}+\hspace{-0.5mm}A(I-G_{k|k-1}^{1}C_1)\,\\
&\times\Sigma_{k|k-1}^2\,(I\hspace{-1mm}-\hspace{-1mm}G_{k|k-1}^{1}C_1)^{\top}A^{\top}
\hspace{-1mm}+\hspace{-1mm}A G_{k|k-1}^{1} Q_{v_1} (G_{k|k-1}^{1})^{\top}A^{\top}\hspace{-1mm}+\hspace{-1mm}Q_w.
\end{aligned}
\end{equation}
\end{small}
In addition, by computation, the estimation error of estimator 2 does not involve any control coupling terms.
Hence, we directly obtain \(\Sigma_{k+1|k}^2\), \(\Sigma_{k|k}^2\).
Therefore, \eqref{kf1} is obtained.
\end{proof}

%写成假设或
\subsection{The Nash equilibrium}
This subsection is dedicated to deriving a closed-form Nash equilibrium under the linear strategy assumption.
In view of Definition~\ref{def1}, finding a Nash equilibrium of game $\mathcal G$ corresponds to solving an individual minimization problem for each player with the strategy of the other player fixed, i.e., best-response analysis.

% Based on Definition~\ref{def1} and the optimal state estimates given in Lemma~\ref{kf_lem}, 
% the Nash equilibrium is obtained via a best-response analysis. 
% Specifically, the optimal action of each player is derived by fixing the opponent^{\top}s strategy at a candidate admissible action.
% The optimal linear strategies, i.e., Nash equilibrium is then characterized by a set of Riccati equations.
% corresponding optimal response; repeating this for both players yields a pair of coupled Riccati recursions. A Nash equilibrium is then obtained as a fixed point of these best responses.
% \medskip
Using Lemmas~\ref{lem1} and~\ref{kf_lem}, we first formulate the following optimization problem for $\hat{u}_k$.

\textbf{(P1): Minimization of $\hat{u}_k$ given $\tilde{u}_k^{2,*}$.}
For all $k \in \{0, \dots, N\}$, we assume that the optimal linear state-feedback strategy
\(
\tilde u_k^{2,*}
=
K_k^2\big(\hat x_{k|k-1}^2-\hat x_{k|k-1}^1\big) 
\) is fixed, where $K_k^2$ and $K_k^1$ are fixed, and the optimal state estimation $\hat x_{k|k-1}^1$ and $\hat x_{k|k-1}^2$ are defined in \eqref{kf}. 
Substituting $\tilde u_k^{2,*}$ into the system dynamics \eqref{rs1} yields
\begin{small}
\begin{equation}\label{eq:cl_step1}
x_{k+1}
=
Ax_k + B\hat u_k + B_2 K_k^2\big(\hat x_{k|k-1}^2-\hat x_{k|k-1}^1\big) + w_k.
\end{equation}
\end{small}
To minimize the cost function \(J^1(\{\hat{u}_k\}_{k=0}^N,\{\tilde{u}_k^{2,*}\}_{k=0}^N)\) in \eqref{rs1}, we introduce the following value function
\begin{small}
\begin{equation}\label{eq:v1_clean}
V_k^1
=
\mathbb{E}\!\left[
x_k^\top P_k^1 \hat x_{k|k-1}^1
+
x_k^\top \Phi_k^1\big(\hat x_{k|k-1}^2-\hat x_{k|k-1}^1\big)
\right],
\end{equation}
\end{small}
where the matrices \(P_k^1\) and \(\Phi_k^1\) satisfy the backward recursions
\begin{small}
\begin{equation}\label{eq:Ric1_clean}
\left\{
\begin{aligned}
&P_k^1 = (A + B K_k^1)^\top P_{k+1}^1 (A + B K_k^1) + Q_1 + (K_k^1)^\top \Gamma_1 K_k^1,\\
&\Phi_k^1 = (A + B_2 K_k^2)^\top \Phi_{k+1}^1 (A + B_2 K_k^2) + (K_k^2)^\top R_1 K_k^2,\\
&P_{N+1}^1=I,\quad \Phi_{N+1}^1=0.
\end{aligned}
\right.
\end{equation}
\end{small}
By employing \eqref{kf} and \eqref{eq:cl_step1}, we can calculate the one-step variation of the value function \(\Delta V_k^1=V_k^1-V_{k+1}^1\). Summing the resulting identities over \(k\) yields
\begin{align}\label{uu1}
\hat u_k^{*}=K_k^1\hat x_{k|k-1}^1,
\end{align}
where the feedback gain $K_k^1$ satisfies
\begin{small}
\begin{equation}\label{k1}
K_k^1 = -(H_k^1)^{-1} B^\top P_{k+1}^1 A,\quad 
H_k^1 = \Gamma_1 + B^\top P_{k+1}^1 B \succ 0.
\end{equation}
\end{small}

Similarly, we consider an optimization problem for $\tilde{u}_k^2$.
% \noindent

\textbf{(P2): Minimization of $\tilde{u}_k^2$ given $\hat{u}_k^*$.}
Suppose that the optimal linear strategy 
\(
\hat u_k^{*}=K_k^1\hat x_{k|k-1}^1,~ \forall k \in \{0, \dots, N\}
\) is adopted, where \(\hat x_{k|k-1}^1\) is given in \eqref{kf} and \(K_k^1\) is fixed.
Then, substituting it into \eqref{rs1}, the system dynamics can be rewritten as
\begin{small}
\begin{equation}\label{eq:cl_step2}
x_{k+1}
=
Ax_k + B K_k^1\hat x_{k|k-1}^1 + B_2 \tilde u_k^{2} + w_k.
\end{equation}
\end{small}
The associated value function is chosen as
\begin{small}
\begin{equation}\label{eq:v2_clean}
V_k^2
=
\mathbb{E}\!\left[
x_k^\top \Phi_k^2 \hat x_{k|k-1}^1
+
x_k^\top P_k^2\big(\hat x_{k|k-1}^2-\hat x_{k|k-1}^1\big)
\right],
\end{equation}
\end{small}
where the matrices \(P_k^2\) and \(\Phi_k^2\) satisfy
\begin{small}
\begin{equation}\label{eq:Ric2_clean}
\left\{
\begin{aligned}
&\Phi_k^2 = (A + B K_k^1)^\top \Phi_{k+1}^2 (A + B K_k^1) + (K_k^1)^\top \Gamma_2 K_k^1 + Q_2,\\
&P_k^2 = (A + B_2 K_k^2)^\top P_{k+1}^2 (A + B_2 K_k^2) + Q_2 + (K_k^2)^\top R_2 K_k^2,\\
&\Phi_{N+1}^2=I,\quad P_{N+1}^2=0.
\end{aligned}
\right.
\end{equation}
\end{small}
To minimize \(J^2(\{\hat{u}_k^*\}_{k=0}^N,\{\tilde{u}_k^{2}\}_{k=0}^N)\) in \eqref{rs1}, we compute a one-step variation \(\Delta V_k^2=V_k^2-V_{k+1}^2\), the optimal linear strategy is obtained as
\begin{align}\label{uu2}
\tilde u_k^{2,*}=K_k^2\big(\hat x_{k|k-1}^2-\hat x_{k|k-1}^1\big),
\end{align}
where the gain \(K_k^2\) satisfies
\begin{small}
\begin{equation}\label{kk2}
K_k^2
=
-(H_k^2)^{-1}B_2^\top P_{k+1}^2 A,
\quad
H_k^2 = R_2 + B_2^\top P_{k+1}^2 B_2 \succ 0 .
\end{equation}
\end{small}

Consequently, we can have the following result.
\begin{theorem}\label{thm1}
Consider the system dynamics \eqref{rs1} and the associated cost functions \eqref{jj2}. Suppose that the coupled Riccati equations \eqref{eq:Ric1_clean} and \eqref{eq:Ric2_clean} admit unique solutions. Then Problem~\ref{prob1} is uniquely solvable, and the game $\mathcal G$ admits a unique Nash equilibrium in linear state-feedback form. More specifically, the following statements hold.
\begin{enumerate} % 移除列表左侧缩进，使编号贴近左边界
\item The feedback gains $K_k^1$ and $K_k^2$
    are uniquely determined by \eqref{k1} and \eqref{kk2}, respectively, for all $ k \in \{0, \dots, N\}$.
\item For all $ k \in \{0, \dots, N\}$, the optimal linear strategies for players satisfy 
\begin{align}\label{thm_utilde}
\hat u_k^{*}&=K_k^1\hat x_{k|k-1}^1,\notag\\
\tilde u_k^{2,*}&=K_k^2\big(\hat x_{k|k-1}^2-\hat x_{k|k-1}^1\big),
\end{align}
and the induced action sequences $\{u_k^{1,*}\}_{k=0}^N$ and $\{u_k^{2,*}\}_{k=0}^N$ satisfy
    \begin{align}
        u_k^{1,*} &= [I_{m_1},\,0]\,\hat u_k^{*}, \notag\\
        u_k^{2,*} &= [0,\,I_{m_2}]\,\hat u_k^{*} + \tilde u_k^{2,*}. \label{thm_u1}
    \end{align}
Moreover, the pair $(\{u_k^{1,*}\}_{k=0}^N, \{u_k^{2,*}\}_{k=0}^N)$ constitutes the Nash equilibrium of game $\mathcal G$.
\item The optimal costs are given by
\begin{small}
\begin{equation*}
\begin{aligned} 
&J^1\big(\{u_k^{1,*}\}_{k=0}^N,\{u_k^{2,*}\}_{k=0}^N\big)\\
&\hspace{-1mm}=\hspace{-1mm} \mathbb{E}\left[ x_0^{\top} P_0^1 \hat{x}_{0|-1}^1 \hspace{-1mm}+\hspace{-1mm} x_0^{\top} \Phi_0^1 (\hat{x}_{0|-1}^2\hspace{-1mm} -\hspace{-1mm} \hat{x}_{0|-1}^1) \right]\\
&\hspace{-1mm}+\hspace{-1mm}\operatorname{Tr}[\Sigma_{N+1|N}^1 P_{N+1}^1] \hspace{-1mm}+\hspace{-1mm} \sum_{k=0}^N \operatorname{Tr} \Bigl[ \Sigma_{k|k-1}^1 ( A^{\top} P_{k+1}^1 A G_{k|k-1}^{1} C_1\\
&+ Q_1) \hspace{-1mm} + \hspace{-1mm}(\Sigma_{k|k-1}^1\hspace{-1mm} - \hspace{-1mm}\Sigma_{k|k-1}^2) \bigl[ (K_k^2)^{\top} B_2^{\top} P_{k+1}^1 A G_{k|k-1}^{1} C_1  \\
&\hspace{-1mm}-\hspace{-1mm}(A + B_2 K_k^2)^{\top} \Phi_{k+1}^1 A G_{k|k-1}^{1} C_1 \bigr] \\
& + \Sigma_{k|k-1}^2 ( A^{\top} \Phi_{k+1}^1 A G_{k|k-1}^2 C - A^{\top} \Phi_{k+1}^1 A G_{k|k-1}^{1} C_1 ) \Bigr], 
\end{aligned}
\end{equation*}
\end{small}
\begin{small}
\begin{equation}
\begin{aligned} 
&J^{2}(\{u_k^{1,*}\}_{k=0}^N,\{u_k^{2,*}\}_{k=0}^N)\\
&= \mathbb{E}[ x_0^{\top} \Phi_0^2 \hat{x}_{0|-1}^1 \hspace{-0.5mm}+\hspace{-0.5mm} x_0^{\top} P_0^2 (\hat{x}_{0|-1}^2\hspace{-0.5mm} -\hspace{-0.5mm} \hat{x}_{0|-1}^1)]\hspace{-1mm}\\
&+\hspace{-1.2mm}\operatorname{Tr}[\Sigma_{N+1|N}^1 \Phi_{N+1}^2]\hspace{-1.2mm}+\hspace{-1.2mm} \sum_{k=0}^N \operatorname{Tr}\Bigl[\Sigma_{k|k-1}^2(A^{\top}P_{k+1}^2AG_{k|k-1}^2C)\\
&+\Sigma_{k|k-1}^1 ( A^{\top} (\Phi_{k+1}^2-P_{k+1}^2) A G_{k|k-1}^{1} C_1+Q_2)\Bigl].\label{oos1}
\end{aligned}
\end{equation}
\end{small}
\end{enumerate}
\end{theorem}

\begin{proof}
Please refer to Appendix~A.
\end{proof}

\begin{remark}
% \textcolor{red}{
Under the linear strategies, the closed-form of the Nash equilibrium \eqref{thm_u1} and the associated optimal costs \eqref{oos1} are explicitly characterized through the best-response analysis. It follows that Problem \ref{prob1} admits a unique Nash equilibrium solution.
% }
\end{remark}

\subsection{Comparison of equilibrium costs under symmetric and asymmetric information}

In this subsection, we characterize the impact of information structure on the costs of players by comparing the total costs under different information structures.

 % together with the backward induction method.

% \textcolor{red}{
Specifically, we consider the symmetric information
and the asymmetric information case, respectively.
% }

1)~\textbf{Symmetric information structure:}
% Consider the 
Given the system \eqref{s1} and cost function \eqref{cs}, assume that the controllers have access to the same information, i.e., a symmetric information structure:
\begin{align}\label{sis1}
\mathcal{F}_k^1=\mathcal{F}_k^2=\sigma\Big\{ \{y_m^1\}_{m=0}^{k-1}, \{u_m^1\}_{m=0}^{k-1}, \{y_m^2\}_{m=0}^{k-1}, \{u_m^2\}_{m=0}^{k-1}\Big\}.
\end{align}
Therefore, both players use the same state estimate, and for notational convenience, we retain the notation \(\hat{x}_{k|k-1}^2\).

% the state estimation
% are $\hat{x}_{k|k-1}^2$
% Hence, $u_k^1$, $u_k^2$ are subject to $\mathcal{F}_k^1$-measurable and $\mathcal{F}_k^2$-measurable, respectively.
% Consider 
To solve Problem \ref{prob1}, the optimal actions for the two players are derived as 
\begin{align}\label{ssk1}
u_k^{1,*}=K_k^1\hat{x}_{k|k-1}^2,\quad u_k^{2,*}=K_k^2\hat{x}_{k|k-1}^2, \quad \forall k \in \{0, \dots, N\}.
\end{align}
where $\hat{x}_{k|k-1}^2$ is given in \eqref{kf}.
Then optimal cost functions are represented as
\begin{small}
\begin{align}\label{sj}
&\hat{J}^{1}(\{u_k^{1,*}\}_{k=0}^N,\{u_k^{2,*}\}_{k=0}^N)\notag\\
&=\mathbb{E}[ x_0^{\top} P_{0}^1 \hat{x}_{0|-1}^2] +\operatorname{Tr}\left[\Sigma_{N+1|N}^2 P_{N+1}^1\right] \notag\\
&\hspace{-1mm}+ \hspace{-1mm}\sum_{k=0}^N \operatorname{Tr}\left[ \Sigma_{k|k-1}^2 (A^{\top} P_{k+1}^1 A G_{k|k-1}^2 C \hspace{-1mm}+ \hspace{-1mm}Q_1) \right],
\end{align}
\end{small}
where \(\Sigma_{k|k-1}^2\), \(G_{k|k-1}^2\) are given in \eqref{kf1}.

% \hat{J}^{2,*} &= \mathbb{E}\left[ x_0^{\top} P_{0}^2 \hat{x}_{0|-1}^2 \right] 
%      + \sum_{k=0}^N \mathbb{E}\left[ \Sigma^2 (A^{\top} P_{k+1}^2 A G C + Q_2) \right]\notag\\
%      &+\operatorname{Tr}\left[\Sigma^2 P_{N+1}^2\right]. 

2)~\textbf{Asymmetric information structure:}
% Consider the game \(\mathcal G\) under an asymmetric information structure, where the information sets of the two players are defined in \eqref{infor1}. 
% Substituting the optimal control sequence \eqref{thm_u1}, we suppose that the feedback gains \(K_k^1\) and \(K_k^2\), together with the Riccati matrices \(P_k^1\) and \(\Phi_k^1\) in \eqref{ssk1}-\eqref{sj}, are identical to those associated with symmetric information structure. 
% According to Theorem~\ref{prob1}, the resulting difference in the optimal costs is given by
Recalling that the asymmetric information structure of the two players is defined in \eqref{infor1}. 
In this case, the Nash equilibrium and the optimal cost functions have been given in \eqref{thm_u1}-\eqref{oos1}.

% The Nash equilibrium and the optimal cost functions under asymmetric information structure are given in \eqref{u2}-\eqref{jj2}.
It is worth mentioning that the Riccati matrices in \eqref{sj} are approximately equal to those in \eqref{jj2}.
Therefore, from \eqref{jj2} and \eqref{sj}, the difference between the associated optimal costs is given by
\begin{small}
\begin{equation}\label{jj3}
\begin{aligned}
&J^{1}(\{u_k^{1,*}\}_{k=0}^N,\{u_k^{2,*}\}_{k=0}^N)-\hat{J}^{1}(\{u_k^{1,*}\}_{k=0}^N,\{u_k^{2,*}\}_{k=0}^N)\\
&=\mathbb{E} \left[ (\hat{x}_{0|-1}^2 - \hat{x}_{0|-1}^1)^\top \Phi_0^1 (\hat{x}_{0|-1}^2 - \hat{x}_{0|-1}^1) \right]+\sum_{k=0}^N \Delta_k,
\end{aligned}
\end{equation}
\end{small}
with 
\small
\begin{align}\label{delta}
\Delta_k&=Tr \Bigl[ \Sigma_{k|k-1}^1 ( A^{\top} P_{k+1}^1 A G_{k|k-1}^{1} C_1 + Q_1) \notag\\
&\quad+ \Sigma_{k|k-1}^2 ( A^{\top} \Phi_{k+1}^1 A G_{k|k-1}^{2} C - A^{\top} \Phi_{k+1}^1 A G_{k|k-1}^{1} C_1 ) \notag\\
&\quad- \Sigma_{k|k-1}^2 (A^{\top} P_{k+1}^1 A G_{k|k-1}^{2} C + Q_1)\Bigl].
\end{align}

For simplicity, we define $\Delta \Sigma_{k|k} \triangleq \Sigma_{k|k}^{1}-\Sigma_{k|k}^2$, and $\Delta \Sigma_{k|k-1} \triangleq \Sigma_{k|k-1}^{1}-\Sigma_{k|k-1}^2$. 

To determine the sign of \eqref{jj3}, we first establish in the following lemma a monotonicity relation between the error covariance matrices and the information structure.

\begin{lemma}\label{lem2}
For the optimal estimates $\hat{x}_{k|k-1}^1$ and $\hat{x}_{k|k-1}^2$ in Lemma~\ref{kf_lem}, the corresponding estimation error covariance matrices satisfy
\begin{align}
\Sigma_{k|k-1}^{1}\succeq \Sigma_{k|k-1}^{2}.
\end{align}
\end{lemma}

\begin{proof}
Since $\mathcal F_k^1 \subset \mathcal F_k^2$ and \(d_k=\hat{x}_{k|k-1}^2-\hat{x}_{k|k-1}^1\), using the tower property of conditional expectation \cite{Durrett2019}, we have
\begin{align}
\mathbb E[d_k\mid \mathcal F_k^1]
=
\mathbb E[\hat x_{k|k-1}^2-\hat x_{k|k-1}^1\mid \mathcal F_k^1]
=0,
\end{align}
Let
\(
e_k^2\triangleq x_k-\hat x_{k|k-1}^2\) and $d_k\triangleq \hat x_{k|k-1}^2-\hat x_{k|k-1}^1$, 
so that $e_k^1=e_k^2+d_k$.
By Lemma~\ref{lem1}, $\mathbb E[e_k^2\mid \mathcal F_k^2]=0$.
Since $d_k$ is $\mathcal F_k^2$-measurable, we have
\(
\mathbb E[e_k^2 d_k^\top\mid \mathcal F_k^2]=0,
\)
and hence, by the tower property,
\(
\mathbb E[e_k^2 d_k^\top\mid \mathcal F_k^1]=0\) and
\(\mathbb E[d_k(e_k^2)^\top\mid \mathcal F_k^1]=0.
\)
Therefore,
\begin{align}
\Sigma_{k|k-1}^1
&= \mathbb E[e_k^1(e_k^1)^\top\mid \mathcal F_k^1] \notag\\
&= \mathbb E[e_k^2(e_k^2)^\top\mid \mathcal F_k^1]
 + \mathbb E[d_k d_k^\top\mid \mathcal F_k^1].
\label{eq:cov_split}
\end{align}
Moreover,
\(
\mathbb E[e_k^2(e_k^2)^\top\mid \mathcal F_k^1]
=
\mathbb E[\Sigma_{k|k-1}^2\mid \mathcal F_k^1].
\)
Under the linear--Gaussian Kalman setting, $\Sigma_{k|k-1}^2$ is deterministic; thus
\(
\mathbb E[\Sigma_{k|k-1}^2\mid \mathcal F_k^1]=\Sigma_{k|k-1}^2.
\)
Substituting this into \eqref{eq:cov_split} yields
\begin{align}
\Delta \Sigma_{k|k-1}
&\triangleq \Sigma^{1}_{k|k-1}-\Sigma^{2}_{k|k-1}\notag\\
&= \mathbb{E}\!\left[d_k d_k^{\top}\mid \mathcal F_k^1\right]
= \mathrm{Cov}(d_k\mid \mathcal F_k^1)\succeq 0,
\end{align}
which proves that $\Sigma_{k|k-1}^1 \succeq \Sigma_{k|k-1}^2$.
\end{proof}

By leveraging the monotonicity of the error covariance matrices established above, we show below that the equilibrium cost of the corresponding player is monotonically non-increasing with the richness of the information set.

\begin{theorem}
\label{thm2}
Applying Lemma~\ref{lem2}, the equilibrium cost difference given in \eqref{jj3} between the symmetric and asymmetric information structures satisfies
\begin{small}
\begin{equation}\label{eq:gap_decomp_final}
\begin{aligned}
&J^{1}(\{u_k^{1,*}\}_{k=0}^N,\{u_k^{2,*}\}_{k=0}^N) - \hat J^{1}(\{u_k^{1,*}\}_{k=0}^N,\{u_k^{2,*}\}_{k=0}^N)\\
&= \mathbb{E} \left[ (\hat{x}_{0|-1}^2 - \hat{x}_{0|-1}^1)^\top \Phi_0^1 (\hat{x}_{0|-1}^2 - \hat{x}_{0|-1}^1) \right] \\
&\quad+ \sum_{k=0}^{N} \mathbb{E} \left[ \operatorname{Tr} \left( Q_1 (\Sigma_{k|k-1}^1\hspace{-0.5mm} -\hspace{-0.5mm} \Sigma_{k|k-1}^2) \right) \right]\hspace{-0.5mm}-\hspace{-0.5mm} \sum_{k=0}^{N} \mathbb{E} \left[ \operatorname{Tr} \left( \Delta P_{k+1}^1 \Delta_k^u \right) \right]\\
&\quad+ \sum_{k=0}^{N} \mathbb{E} \left[ \operatorname{Tr} \left( \Delta P_{k+1}^1 (\Sigma_{k+1|k}^1 - \Sigma_{k+1|k}^2) \right) \right] \\
& \quad+ \sum_{k=0}^{N} \mathbb{E} \left[ \operatorname{Tr} \left( \Delta P_k^1 (\Sigma_{k|k-1}^1 - \Sigma_{k|k-1}^2) \right) \right] \geq 0,
\end{aligned}
\end{equation}
\end{small}
where $\Delta P_k^1 \triangleq P_k^1 - \Phi_k^1$ and $
\Delta_k^{u} = B_{2} K_k^{2}\big(\Sigma_{k|k-1}^{2}-\Sigma_{k|k-1}^{1}\big)\big(K_k^{2}\big)^{\top} B_{2}^{\top}$.
% Furthermore, $J^{1,*} - \hat J^{1,*} \ge 0$, which implies $\hat J^{1,*} \le J^{1,*}$.
\end{theorem}

\begin{proof}
See Appendix \ref{app2}.
\end{proof}

\begin{remark}
% \textcolor{blue}{
Beyond the characterization of the Nash equilibrium for Problem~\ref{prob1}, we further investigate how the information structure affects the cost for the corresponding player. Utilizing the tower property of conditional expectation, the monotonicity of the error covariance with respect to the information set is established in Lemma~\ref{lem2}. By exploiting the positive semidefiniteness of the weighting matrices, we subsequently demonstrate that the equilibrium cost is also monotone with respect to the information structure, as evidenced by \eqref{eq:gap_decomp_final}.
\end{remark}

% Although Theorem 1 shows that the optimal action feedback gain can be computed independently of the estimation gain, it follows from Lemma 1 that the estimation gain depends on the control gain. To resolve this coupling issue, we will investigate the steady-state solution in the subsequent sections.}
% \end{remark}

\subsection{Steady-State Kalman filter and Nash Equilibrium}
 % obtained in Lemma \ref{kf_lem} and Theorem~\ref{thm1}.

% To address the coupling difficulty of estimation gain and state-feedback gain, 
In this subsection, we investigate the convergence of the Riccati matrices and derive the corresponding steady-state Kalman filter solutions and Nash equilibrium strategies.

First, we make the following standard assumption.
\begin{assumption}
\label{ass2}
The pairs \((A, B)\) and \((A, B_2)\) are stabilizable, and the pairs
\((A, Q_1^{1/2})\) and \((A, Q_2^{1/2})\) are detectable. 
\end{assumption}
According to Assumption~\ref{ass2}, suppose that the coupled algebraic Riccati equations admit positive-definite stabilizing solutions \(P^1\succ0\), \(P^2\succ0\), \(\Phi^1\succ0\), and \(\Phi^2\succ0\), which are given by
\begin{equation}\label{Ric2}
\left\{
\begin{aligned}
&P^1 = (A + BK^1)^\top P^1 (A + BK^1) + Q_1 + (K^1)^\top \Gamma_1 K^1,\\
&P^2 = (A + B_2K^2)^\top P^2 (A + B_2K^2) + Q_2 + (K^2)^\top R_2 K^2,\\
&\Phi^1 = (A + B_2K^2)^\top \Phi^1 (A + B_2K^2) + (K^2)^\top R_1 K^2,\\
&\Phi^2 = (A + BK^1)^\top \Phi^2 (A + BK^1) + (K^1)^\top \Gamma_2 K^1 + Q_2,
\end{aligned}
\right.
\end{equation}
where
\begin{align*}
&K^1 = -(H^1)^{-1}B^\top P^1 A,\qquad H^1 = \Gamma_1 + B^\top P^1 B,\\
&K^2 = -(H^2)^{-1}B_2^\top P^2 A,\qquad H^2 = R_2 + B_2^\top P^2 B_2, 
\end{align*}
together with the corresponding estimation error covariance satisfy
\begin{small}
\begin{equation}\label{skf2}
\begin{aligned}
\Sigma^1
&=(A(I-G^{1}C_1)\hspace{-1mm}+\hspace{-1mm}B_2K^2)(\Sigma^1-\Sigma^2)
(A(I-G^{1}C_1)\hspace{-1mm}+\hspace{-1mm}B_2K^2)^{\top}  \\
&\hspace{-1mm}+\hspace{-1mm}A(I-G^{1}C_1)\Sigma^2(I-G^{1}C_1)^{\top}A^{\top}
\hspace{-1mm}+\hspace{-1mm}A G^{1} Q_{v_1} (G^{1})^{\top}A^{\top}\hspace{-1mm}+\hspace{-1mm}Q_w,\\
G^{1}&= \Sigma^{1} C_1^{\top}\big(C_1 \Sigma^{1} C_1^{\top}+Q_{v_1}\big)^{-1}.
\end{aligned}
\end{equation}
\end{small}

In view of Assumption \ref{ass2}, we present the following convergence analysis of the coupled matrix recursions.
\begin{lemma}
\label{lem:riccati_conv}
Under Assumption 2, as the horizon tends to infinity, the Riccati recursions $\{P_k^i, \Phi_k^i\}$, feedback gains $\{K_k^i\}$, and estimation error covariances $\{\Sigma_{k|k-1}^i\}$ for $i \in \{1,2\}$ converge to their respective unique steady-state stabilizing solutions defined in \eqref{Ric2}.
\end{lemma}

\begin{proof}
Under Assumption \ref{ass2}, the solutions \(P^1\), \(P^2\), \(\Phi^1\), \(\Phi^2\),  \(\Sigma^1\) and \(\Sigma^2\) exist uniquely.
To compute \(P^1\), \(P^2\), \(\Phi^1\) and \(\Phi^2\) in \eqref{Ric2}, we consider the following forward iterations
\begin{equation}\label{Ric0}
\left\{
\begin{aligned}
P_{k+1}^1
&=
(A+BK_k^1)^\top P_{k}^1 (A+BK_k^1)
+Q_1+(K_k^1)^\top \Gamma_1 K_k^1, \\
P_{k+1}^2
&=
(A+B_2K_k^2)^\top P_{k}^2 (A+B_2K_k^2)
+Q_2+(K_k^2)^\top R_2 K_k^2,\\
\Phi_{k+1}^1
&=
(A+B_2K_k^2)^\top \Phi_{k}^1 (A+B_2K_k^2)
+(K_k^2)^\top R_1 K_k^2, \\
\Phi_{k+1}^2
&=
(A+BK_k^1)^\top \Phi_{k}^2 (A+BK_k^1)
+(K_k^1)^\top \Gamma_2 K_k^1 + Q_2,
\end{aligned}
\right.
\end{equation}
where
\begin{align*}
K_k^1& = -(H_k^1)^{-1}B^\top P_{k}^1A,\quad H_k^1= \Gamma_1 + B^\top P_{k}^1 B,\\
K_k^2&= -(H_k^2)^{-1}B_2^\top P_{k}^2A,
\quad H_k^2= R_2 + B_2^\top P_{k}^2 B_2
\end{align*}
with initial values $P_0^1=P_0^2=\Phi_0^1=\Phi_0^2=aI (a>0)$. 

When the time $k$ is large enough, under Assumptions \ref{ass1} and \ref{ass2}, the coupled backward Riccati recursions \eqref{eq:Ric1_clean}, \eqref{eq:Ric2_clean} and the coupled forward Riccati recursions \eqref{Ric0}
will converge to the same algebraic Riccati equation \eqref{Ric2}. In other words, the backward Riccati equations \eqref{eq:Ric1_clean}, \eqref{eq:Ric2_clean} have been transformed into forward Riccati equations, which can be computed simultaneously.
Under Assumption~\ref{ass2}, the pair \((A,B)\) is stabilizable and \((A,Q_1^{1/2})\) is detectable, while the pair \((A,B_2)\) is stabilizable and \((A,Q_2^{1/2})\) is detectable. Hence, by standard convergence theory for discrete-time algebraic Riccati equations, the finite-horizon Riccati sequences \(P_k^1\) and \(P_k^2\) converge, as the horizon tends to infinity, to the unique stabilizing solutions \(P^1\) and \(P^2\) of \eqref{Ric2}. 

Since \(\Gamma_1\succ 0\), \(R_2\succ 0\), it follows that \(H_k^1 \succ 0\) and \(K_k^2 \succ 0\) for all $k \in \{0, \dots, N\}$, and hence \(H_k^1\) and \(H_k^2\) are invertible. Therefore, by continuity of matrix multiplication, the feedback gains \(K_k^1\) and \(K_k^2\) will converge to \(K^1\) and \(K^2\) in \eqref{Ric2}. Moreover, since the recursion for $\Sigma_{k|k-1}^1$ depends on $K_k^2$, the convergence of $K_k^2$ further implies the convergence of $\Sigma_{k|k-1}^1$.

In addition, the recursions for \(\Phi_k^1\) and \(\Phi_k^2\) are linear matrix recursions driven by the closed-loop matrices \(A+B_2K_k^2\) and \(A+BK_k^1\), respectively. As \(K_k^1\), \(K_k^2\) converge to \(K^1\) and \(K^2\), the coefficients in the above recursions converge to the
constant matrices associated with the limiting closed-loop systems.
Moreover, because \(A+BK^1\) and \(A+B_2K^2\) are Schur stable,
the corresponding algebraic equations admit unique solutions \(\Phi^1\) and \(\Phi^2\). Therefore, \(\Phi_k^1\) and \(\Phi_k^2\) converges to \(\Phi^1\) and \(\Phi^2\). This completes the proof.
\end{proof}

According to the above convergence analysis, the steady-state Nash equilibrium strategies can be characterized as follows.
\begin{theorem}
\label{thm:ss_NE}
Given the convergence results in Lemma~\ref{lem:riccati_conv}, let $P^i$, $\Phi^i$, $K^i$ and $\Sigma^i$ ($i=1,2$) be the positive-definite stabilizing solutions to \eqref{Ric2} and \eqref{skf2}, respectively. Then, the steady-state Nash equilibrium strategies are uniquely determined by:
\begin{align}
u_k^{1,*} &= [I_{m_1}, 0]K^1\hat{x}_{k|k-1}^1, \notag\\
u_k^{2,*} &=[I_{m_1}, 0]K^1\hat{x}_{k|k-1}^1+[0, I_{m_2}]K^2(\hat{x}_{k|k-1}^2-\hat{x}_{k|k-1}^1).
\label{thm_u2}
\end{align}
Moreover, these equilibrium strategies render the closed-loop system mean-square bounded. The corresponding steady-state estimations are given by
\begin{equation}\label{skf1}
\left\{
\begin{aligned}
\hat{x}_{k|k}^1&=\hat{x}_{k|k-1}^1+G^1(y_k^1-C_1\hat{x}_{k|k-1}^1),\\
\hat{x}_{k|k-1}^1&=A\hat{x}_{k-1|k-1}^1+B\hat{u}_{k-1},\\
\hat{x}_{k|k}^2&=\hat{x}_{k|k-1}^2+G^2(y_k-C\hat{x}_{k|k-1}^2),\\
\hat{x}_{k|k-1}^2&=A\hat{x}_{k-1|k-1}^2+B\hat{u}_{k-1}+B_2\tilde{u}_{k-1}^2.
\end{aligned}
\right.
\end{equation}
\end{theorem}

\begin{proof}
From Theorem~\ref{thm1}, the equilibrium strategies depend on the time-varying solutions of the Riccati recursions.
By Lemma~\ref{lem:riccati_conv}, the state-feedback gains converge to constant matrices \(K^1\) and \(K^2\).
Substituting these steady-state gains into the equilibrium strategies \eqref{thm_u1} in Theorem~\ref{thm1} yields the steady-state Nash equilibrium in \eqref{thm_u2}.
Moreover, the gains \(K^1\) and \(K^2\) correspond to stabilizing solutions of the Riccati equations, which ensures the stability of the closed-loop system matrix. Therefore, the state process \eqref{rs1} is mean-square bounded.
% \textcolor{red}{
Finally, under the observability condition, the Kalman filtering recursions converge to the steady-state estimates described in \eqref{skf1}. This completes the proof.
% }
\end{proof}

\section{Numerical Example}\label{ne}
In this section, the results are illustrated by a numerical example. 
We consider system \eqref{s1} and cost function \eqref{cs} with the following  parameters
% \begin{small}
\begin{equation}
\begin{aligned}
&n=2, m_1=2, m
_2=2, p_1=2, p_2=2, N=50,\\
&A = \begin{bmatrix}
0.98 & 0.05 \\
0.02 & 0.96
\end{bmatrix},
B_1 = \begin{bmatrix}
0.40 & 0.00 \\
0.00 & 0.30
\end{bmatrix},
B_2 = \begin{bmatrix}
0.35 & 0.00 \\
0.00 & 0.40
\end{bmatrix},\\
&C_1 = \begin{bmatrix}
1 & 0 \\
0 & 0
\end{bmatrix},
C_2 = I_{n},
C = \begin{bmatrix} C_1 \\ C_2 \end{bmatrix}, \\
&Q_w = 0.06 \cdot I_{n}, \Sigma_{0|-1}^1=\Sigma_{0|-1}^2= 0.1I,\\
&R_1 = \operatorname{diag}\big(0.3,\; 0.3\big), R_2 = 0.01 \cdot I_{p_2}, R = \operatorname{diag}(R_1, R_2),\\
&Q_1 = 2.0 \cdot I_{n}, \quad Q_2 = 2.0 \cdot I_{n}, 
% Q_{1f} &= 1.0 \cdot I_{n_x}, \quad Q_{2f} &= 1.0 \cdot I_{n_x}, \\
\quad R_{1} = 2.0 \cdot I_{m_1}, \\
&R_{2} = 2.0 \cdot I_{m_2}, \quad S_{1} = 1.0 \cdot I_{m_2},  \quad S_{2} = 1.0 \cdot I_{m_1},\\
&\Gamma_1=\begin{bmatrix}
S_1 &  \\
    & R_1
\end{bmatrix}, \quad
\Gamma_2=\begin{bmatrix}
S_2 &  \\
    & R_2
\end{bmatrix}.
\end{aligned}
\end{equation}
% \end{small}
By computing the solution of the Riccati equations \eqref{eq:Ric1_clean} and \eqref{eq:Ric2_clean} iteratively, for all $ k \in \{0, \dots, N\}$, the parameter matrices $P_k^1$, $\Phi_k^1$, $P_k^2$, and $\Phi_k^2$ converge to the stable value $P^1$, $\Phi^1$, $P^2$, and $\Phi^2$, respectively, which are shown in Fig.~\ref{p1}-\ref{p2}. 
Furthermore, the optimal feedback gains $K_k^{1}$ and $K_k^{2}$ are illustrated in Fig.~\ref{kk12}.

As shown in Fig~\ref{xxx}, the estimates of the estimator 2 are closer to the real state values than the estimator 1, which implies that the estimator 2 is more accurate. Accordingly, for the horizon $N$, the estimation error covariance $\Sigma_{k|k-1}^1$ is strictly larger than $\Sigma_{k|k-1}^2$, as shown in Fig~\ref{sigmapred}. 

By applying the optimal actions given in \eqref{ssk1} and \eqref{thm_u1}, respectively, we can compare the total costs under the symmetric information and the asymmetric information case. The corresponding results are summarized in Table \ref{tab1}. 
It is observed that the optimal cost for player 1 under asymmetric information is higher than under symmetric information, then we can conclude that richer information reduces the cost of the corresponding player. This validates Theorem \ref{thm2}. 

\begin{figure}[H]
    \centering
\includegraphics[width=0.85\columnwidth]{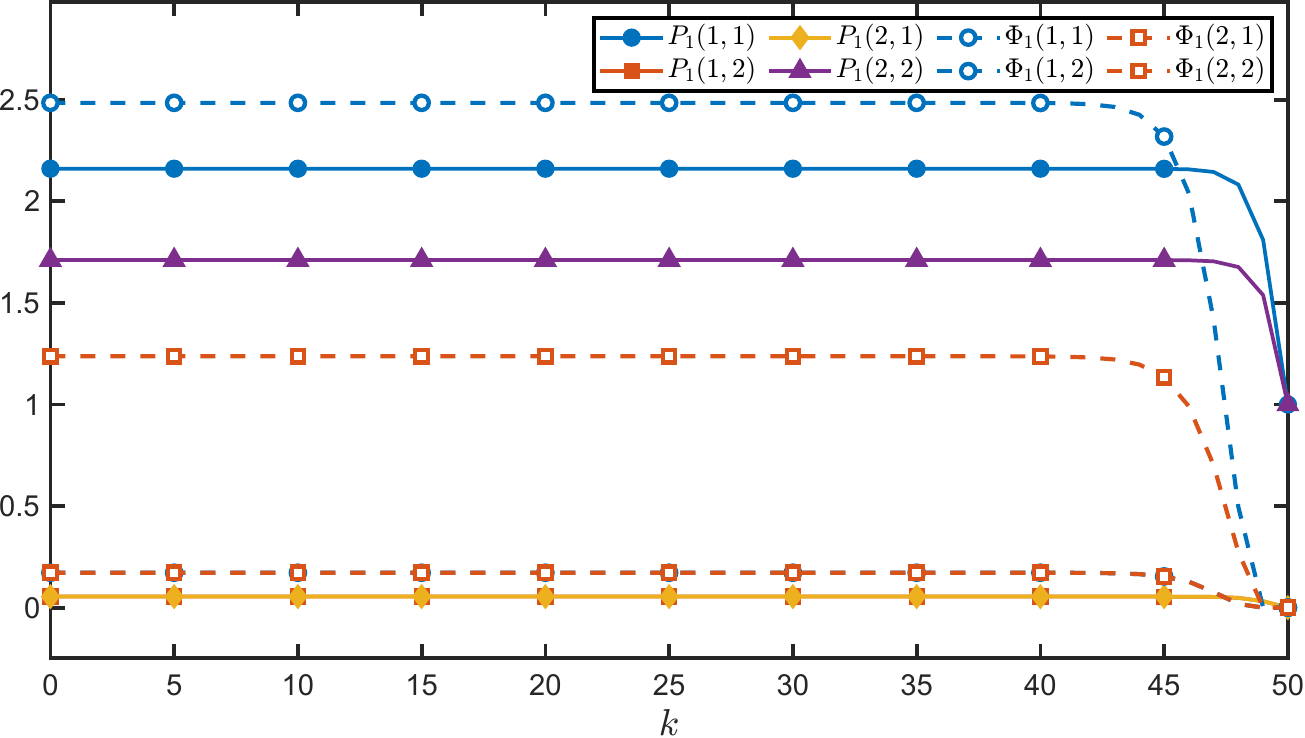}
    \caption{The Riccati matrices $P_k^1$, $\Phi_k^1$}
    \label{p1}
\end{figure}
\vspace{-2.5em}
\begin{figure}[H]
    \centering
\includegraphics[width=0.85\columnwidth]{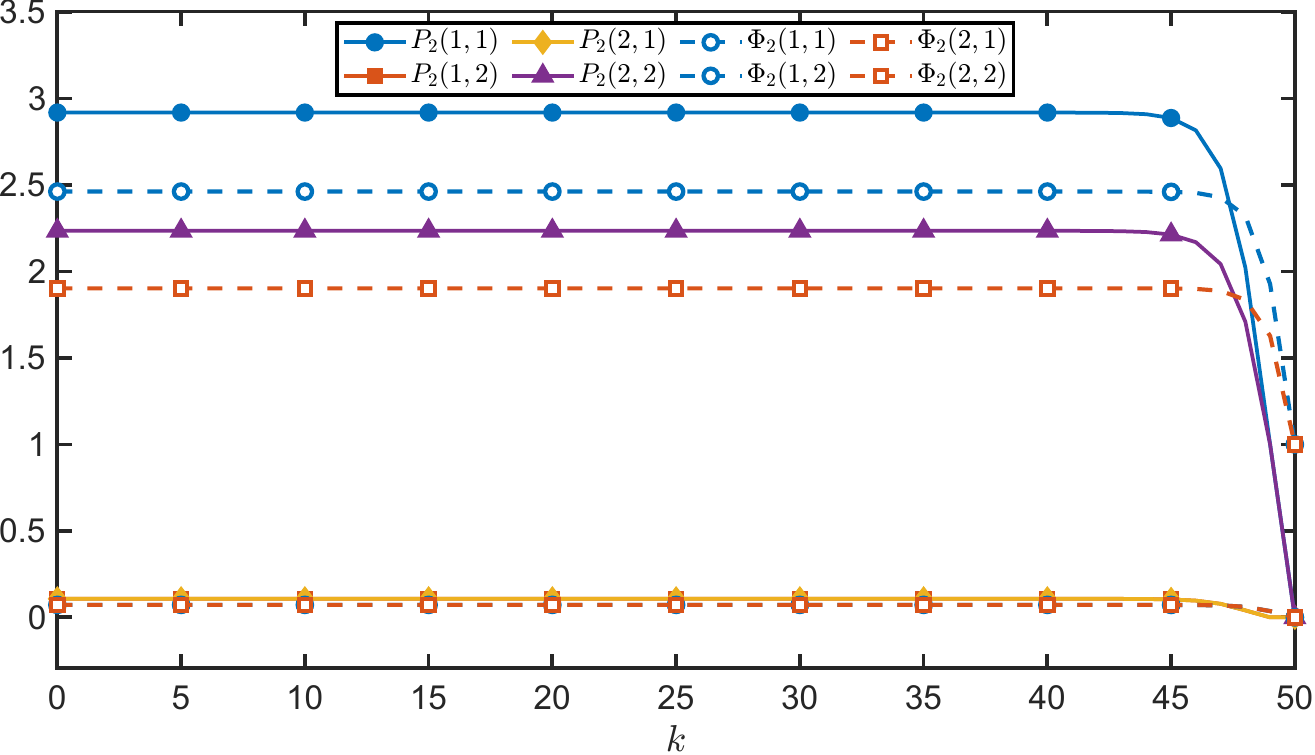}
    \caption{The Riccati matrices $P_k^2$, $\Phi_k^2$}
    \label{p2}
\end{figure}
\vspace{-1em}
\begin{figure}[H]
    \centering
\includegraphics[width=0.85\columnwidth]{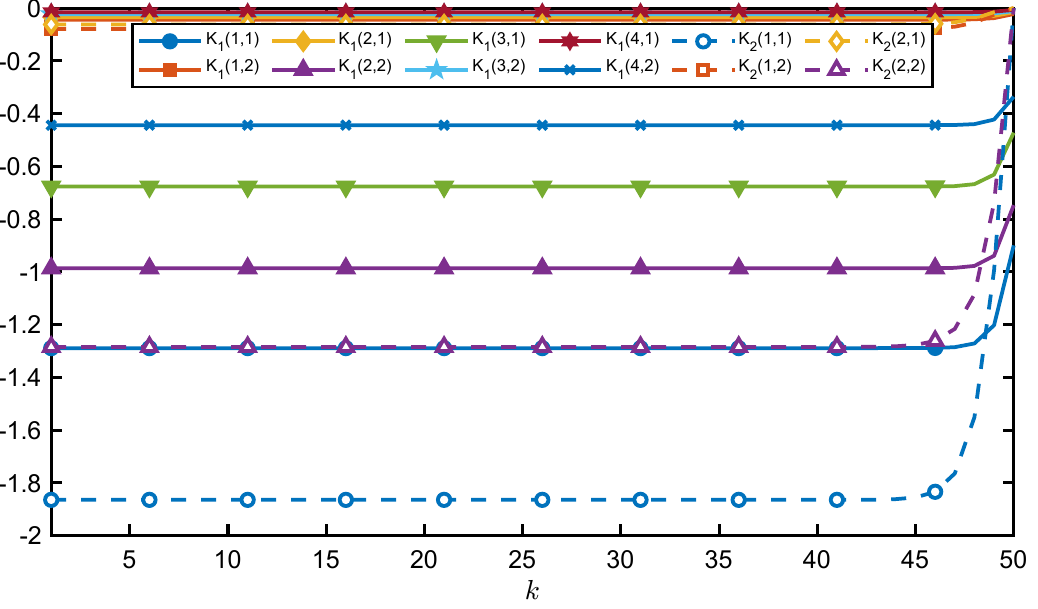}
    \caption{The feedback gains $K_k^1$, $K_k^2$}
    \label{kk12}
\end{figure}
\vspace{-1em}
\begin{figure}[H]
\centering\includegraphics[width=0.85\columnwidth]{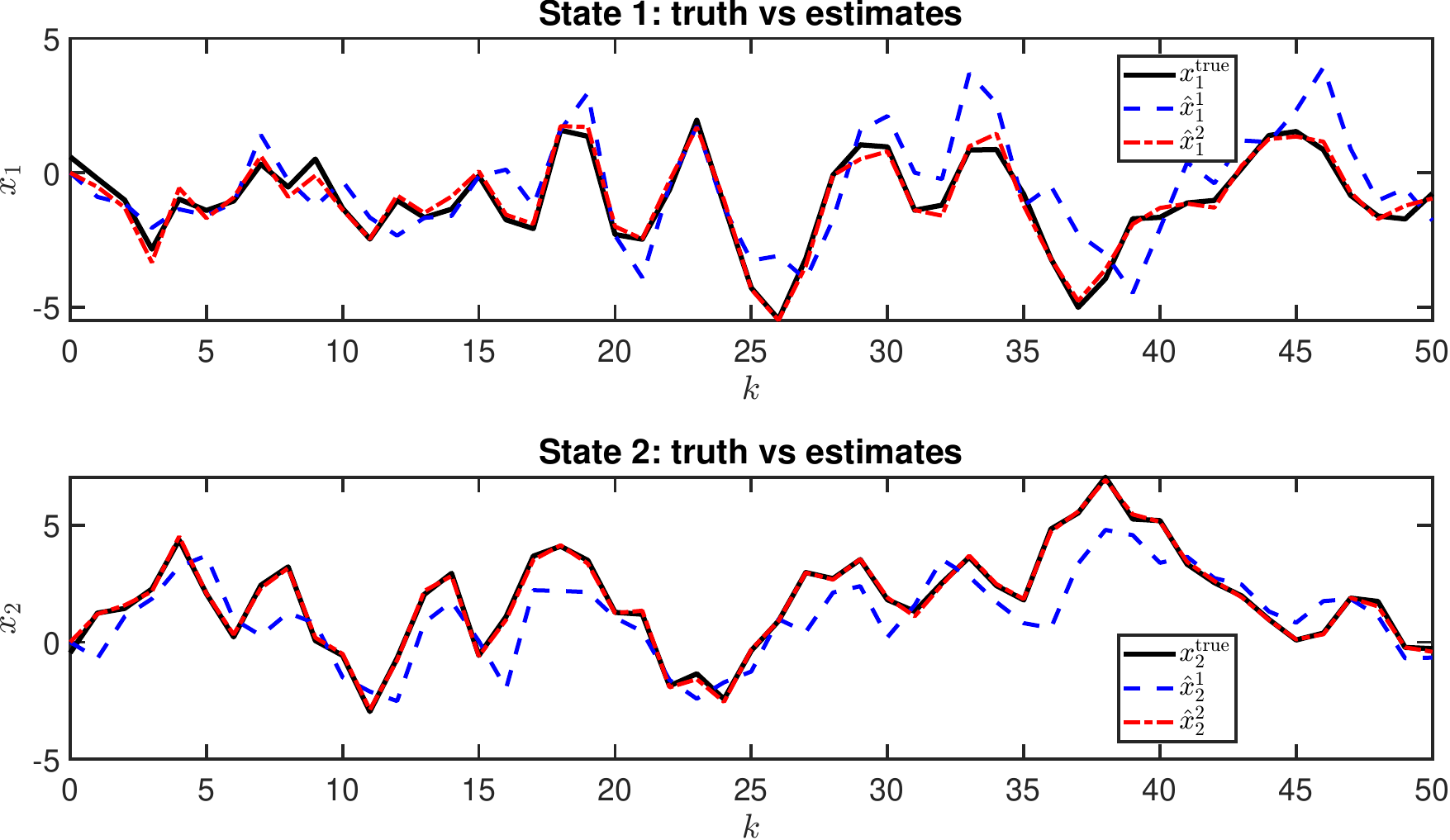}
    \caption{Trajectory of $x_k$, $\hat{x}_{k|k-1}^1$ and $\hat{x}_{k|k-1}^2$}
    \label{xxx}
\end{figure}
\vspace{-1em}
\begin{figure}[H]
    \centering
\includegraphics[width=0.85\columnwidth]{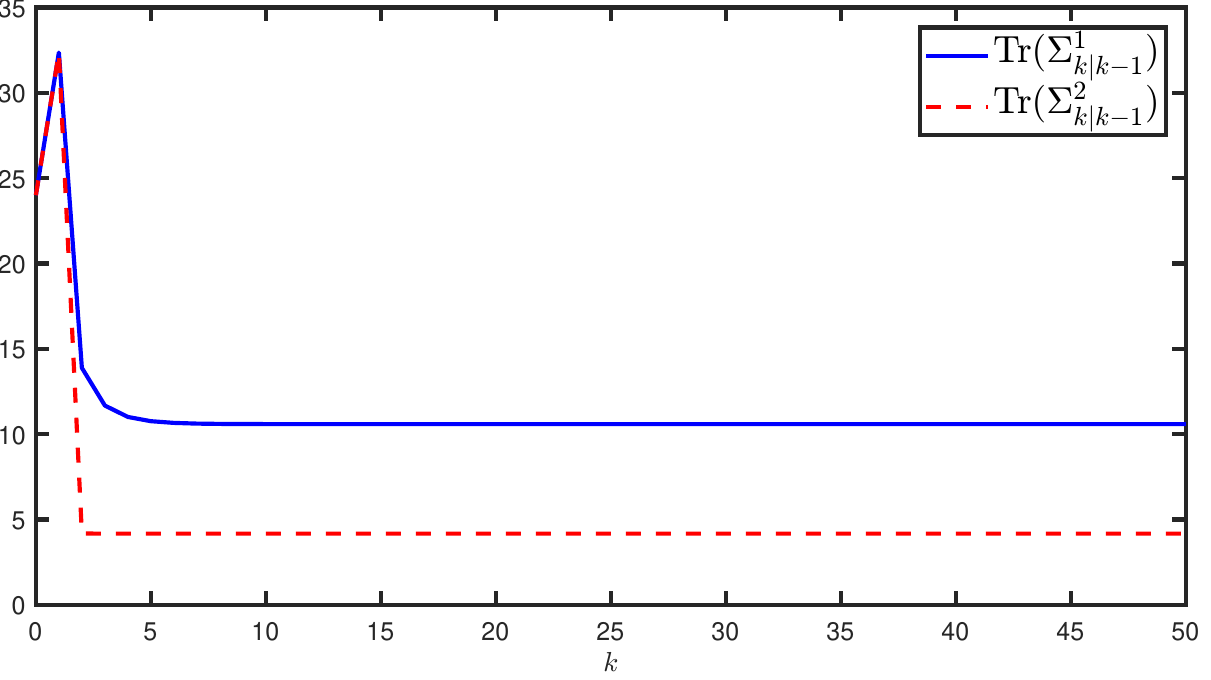}
    \caption{Comparison of $\operatorname{Tr}(\Sigma_{k|k-1}^1)$ and $\operatorname{Tr}(\Sigma_{k|k-1}^2)$}
    \label{sigmapred}
\end{figure}
% \vspace{-1em}
\begin{table}[H]
    \centering
    \caption{Optimal costs under symmetric and asymmetric information structures.}
    \label{tab1}
    \small
    \begin{tabular}{lcc}
        \toprule
        Information Structure & $J^{1}$ & $J^{2}$ \\
        \midrule
        Symmetric information  & 24.4764 & 25.1805 \\
        Asymmetric information & 27.3403 & 28.6230 \\
        \bottomrule
    \end{tabular}
\end{table}

\section{Conclusion}\label{cc}
% \textcolor{blue}{
This article studied noncooperative games under delayed and asymmetric information. By exploiting a common-private information decomposition approach, we characterize the Nash equilibrium via coupled Riccati equations within the class of linear strategies. We theoretically demonstrate that richer information sets help reduce the cost of the associated player. Furthermore, the convergence of the Nash equilibrium solutions is established via a forward iteration scheme under standard stability and detectability conditions. Numerical results validate the theoretical findings, and future work will generalize this framework to multiplayer systems with heterogeneous delay structures.
\appendix

% \section{Appendix A}

\subsection{Proof of Theorem \ref{thm1}}\label{app1}

\begin{proof}
On the one hand, suppose that
$\tilde{u}_k^{2,*}=K_k^2(\hat{x}_{k|k-1}^2-\hat{x}_{k|k-1}^1), ~\forall k \in \{0, \dots, N\}$ is fixed, in which $K_k^2$ satisfies \eqref{kk2}. 
Applying Lemma \ref{lem2}, the difference of the value function defined in \eqref{eq:v1_clean} can be given as
\allowdisplaybreaks
\begin{small}
\begin{align}
&\Delta {V}_k^1 = V_{k}^{1}-V_{k+1}^{1} \nonumber \\
&= \mathbb{E}\big[x_k^{\top} P_k^1 \hat{x}_{k|k-1}^1 + x_k^{\top} \Phi_k^1 (\hat{x}_{k|k-1}^2-\hat{x}_{k|k-1}^1) \big] \label{eq2} \\
&-\mathbb{E}\big[Ax_k\hspace{-0.5mm} + \hspace{-0.5mm}B\hat{u}_k\hspace{-0.5mm} +\hspace{-0.5mm} K_k^{2} (\hat{x}_{k|k-1}^2-\hat{x}_{k|k-1}^1)\hspace{-0.5mm} +\hspace{-0.5mm}w_k\big]^{\top} \nonumber \\
&\times P_{k+1}^1 \big[A\hat{x}_{k|k-1}^1\hspace{-1mm}+\hspace{-1mm} AG_{k|k-1}^{1}C_1(\hat{x}_{k|k-1}^2\hspace{-1mm}-\hspace{-1mm}\hat{x}_{k|k-1}^1)\hspace{-1mm}+\hspace{-1mm}AG_{k|k-1}^{1}v_k^1 \nonumber \\
& + B\hat{u}_k\big]\hspace{-1mm}-\hspace{-1mm}\mathbb{E}\big[Ax_k\hspace{-1mm} + \hspace{-1mm}B\hat{u}_k\hspace{-1mm} + \hspace{-1mm}K_k^{2} (\hat{x}_{k|k-1}^2\hspace{-1mm}-\hspace{-1mm}\hat{x}_{k|k-1}^1)\hspace{-1mm} +\hspace{-1mm}w_k\big]^{\top}\Phi_{k+1}^1 \nonumber \\
&\times\big[(A+B_2K_k^2)(\hat{x}_{k|k-1}^2-\hat{x}_{k|k-1}^1)\hspace{-1mm} \nonumber \\
&+\hspace{-1mm}AG_{k|k-1}^2C(x_k-\hat{x}_{k|k-1}^2)-\hspace{-1mm} AG_{k|k-1}^{1}C_1(\hat{x}_{k|k-1}^2-\hat{x}_{k|k-1}^1)\big] \nonumber \\
&=-\mathbb{E}\big[(\hat{u}_k-K_k^1\hat{x}_{k|k-1}^1)^{\top}(\Gamma_1+B^{\top}P_{k+1}^1B)(\hat{u}_k-K_k^1\hat{x}_{k|k-1}^1)] \nonumber \\
&+(\hat{x}_{k|k-1}^1)^{\top}(P_k^1\hspace{-1mm}-\hspace{-1mm}A^{\top}P_{k+1}^1A+(K_k^1)^{\top}H_k^1K_k^1)\hat{x}_{k|k-1}^1 \nonumber \\
&+(\hat{x}_{k|k-1}^2\hspace{-1mm}-\hspace{-1mm}\hat{x}_{k|k-1}^1)^{\top}(\Phi_k^1\hspace{-1mm}-\hspace{-1mm}(A\hspace{-1mm}+\hspace{-0.5mm}B_2K_k^2)^{\top}\Phi_{k+1}^1(A\hspace{-1mm} \nonumber \\
& +\hspace{-1mm}B_2K_k^2))(\hat{x}_{k|k-1}^2\hspace{-1mm}-\hspace{-1mm}\hat{x}_{k|k-1}^1)\big]\hspace{-1mm}+\hspace{-1mm} \operatorname{tr}\big[ \Sigma_{k|k-1}^1(-A^{\prime} P_{k+1}^1 A G_{k|k-1}^{1} C_1) \nonumber \\
&\hspace{-1mm} +\hspace{-1mm} (\Sigma_{k|k-1}^1 - \Sigma_{k|k-1}^2)\big((A\hspace{-1mm} +\hspace{-1mm} B_2 K_k^2)^{\prime}\Phi_{k+1}^1 A G_{k|k-1}^{1} C_1 \nonumber \\
&\hspace{-1mm}-\hspace{-1mm} (K_k^2)^{\prime} B_2^{\prime} \Phi_{k+1}^1 A G_{k|k-1}^{1} C_1\big)\hspace{-1mm} \nonumber \\
&+\hspace{-1mm}\Sigma_{k|k-1}^2(A^{\top}\Phi_{k+1}^1AG_{k|k-1}^{1}C_1\hspace{-1mm}-\hspace{-1mm}A^{\top}\Phi_{k+1}^1AG_{k|k-1}C)\big].
\end{align}
\end{small}
Using \eqref{eq:Ric1_clean}, it can be obtained that
\begin{small}
\begin{equation}\label{rrv1}
\begin{aligned}
&\Delta {V}_k^1
\hspace{-1mm}=\hspace{-1mm}-\mathbb{E}\big[(\hat{u}_k\hspace{-1mm}-\hspace{-1mm}K_k^1\hat{x}_{k|k-1}^1)^{\top}(\Gamma_1\hspace{-1mm}+\hspace{-1mm}B^{\top}P_{k+1}^1B)(\hat{u}_k\hspace{-1mm}-\hspace{-1mm}K_k^1\hat{x}_{k|k-1}^1)]\hspace{-1mm}\\
&\hspace{-1mm}+\hspace{-1mm}(\hat{x}_{k|k\hspace{-1mm}-1}^1)^{\top}Q_1\hat{x}_{k|k-1}^1\hspace{-1mm}\\
&+\hspace{-1mm}(\hat{x}_{k|k-1}^2\hspace{-1mm}-\hspace{-1mm}\hat{x}_{k|k-1}^1)^{\top}(K_k^2)^{\top}R_1K_k^2(\hat{x}_{k|k-1}^2\hspace{-1mm}-\hspace{-1mm}\hat{x}_{k|k-1}^1)\big]\\
 &\hspace{-1mm} +\hspace{-1mm} (\Sigma_{k|k-1}^1\hspace{-1mm} -\hspace{-1mm} \Sigma_{k|k-1}^2)\big((A\hspace{-1mm} +\hspace{-1mm} B_2 K_k^2)^{\top}\Phi_{k+1}^1 A G_{k|k-1}^{1} C_1 \\
  &\hspace{-1mm}-\hspace{-1mm} (K_k^2)^{\top} B_2^{\prime} \Phi_{k+1}^1 A G_{k|k-1}^{1} C_1\big)\hspace{-1mm}+ \hspace{-1mm}\operatorname{tr}\big[ \Sigma_{k|k-1}^1(-\hspace{-1mm}A^{\top} P_{k+1}^1 A G_{k|k-1}^{1} C_1)\\
  &+\hspace{-1mm}\Sigma_{k|k-1}^2(A^{\top}\Phi_{k+1}^1AG_{k|k-1}^{1}C_1\hspace{-1mm}-\hspace{-1mm}A^{\top}\Phi_{k+1}^1AG_{k|k-1}C)\big].
\end{aligned}
\end{equation}
\end{small}
By adopting \eqref{eec2}, and summing equation \eqref{rrv1} from $k=0$ to $k=N$, we have
{\small
\begin{align}\label{vvv1}
&\sum_{k=0}^N \Delta V_k^1\notag\\
&=\hspace{-1mm}\mathbb{E}\!\left[x_0^{\top} P_0^1 \hat{x}_{0|0}^1
+ x_0^{\top} \Phi_0^1 \bigl(\hat{x}_{0|0}^2-\hat{x}_{0|0}^1\bigr)\right] \notag\\
&\hspace{-1mm}- \mathbb{E}\!\left[x_{N+1}^{\top} P_{N+1}^1 \hat{x}_{N+1|N}^1
+ x_{N+1}^{\top} \Phi_{N+1}^1 \bigl(\hat{x}_{N+1|N}^2-\hat{x}_{N+1|N}^1\bigr)\right] \notag\\
&\hspace{-1mm}=\hspace{-1mm} -\sum_{k=0}^N \mathbb{E}\!\Big[
(\hat{u}_k\hspace{-1mm}-\hspace{-1mm}K_k^1 \hat{x}_{k|k-1}^1)^{\top}
(\Gamma_1+B^{\top}P_{k+1}^1 B)
(\hat{u}_k\hspace{-1mm}-\hspace{-1mm}K_k^1 \hat{x}_{k|k-1}^1)
\Big] \notag\\
&+ \sum_{k=0}^N \operatorname{tr}\!\Big[
\Sigma_{k|k-1}^1\bigl(Q_1 + A^\top P_{k+1}^1 A G_{k|k-1}^{1} C_1\bigr)
\notag\\
&+ (\Sigma_{k|k-1}^1-\Sigma_{k|k-1}^2)\Big(
-(A + B_2 K_k^2)^\top \Phi_{k+1}^1 A G_{k|k-1}^{1} C_1\notag\\
&+ (K_k^2)^\top B_2^\top \Phi_{k+1}^1 A G_{k|k-1}^{1} C_1
\Big)\notag\\
&+ \Sigma_{k|k-1}^2\bigl(
A^{\top}\Phi_{k+1}^1 A G_{k|k-1}^2 C
- A^{\top}\Phi_{k+1}^1 A G_{k|k-1}^{1} C_1
\bigr)
\Big],
\end{align}}
Substituting \(\tilde{u}_k^{2,*}=K_k^2(\hat{x}_{k|k-1}^2-\hat{x}_{k|k-1}^1)\) into \eqref{jj2} and using \eqref{vvv1}, we obtain
\begin{align}
&J^{1}\bigl(\{\hat{u}_k\}_{k=0}^N,\{\tilde{u}_k^{2,*}\}_{k=0}^N\bigr)
= \sum_{k=0}^N \mathbb{E}\!\Big[
x_k^{\top} Q_1 x_k
+ \hat{u}_k^{\top} \Gamma_1 \hat{u}_k\\
&+ (\hat{x}_{k|k-1}^2-\hat{x}_{k|k-1}^1)
\Big] + \mathbb{E}\!\left[x_{N+1}^{\top} P_{N+1}^1 x_{N+1}\right] \notag\\
&\hspace{-1mm}=\hspace{-1mm} -\sum_{k=0}^N \mathbb{E}\!\Big[
(\hat{u}_k\hspace{-1mm}-\hspace{-1mm}K_k^1 \hat{x}_{k|k-1}^1)^{\top}
(\Gamma_1\hspace{-1mm}+\hspace{-1mm}B^{\top}P_{k+1}^1 B)
(\hat{u}_k\hspace{-1mm}-\hspace{-1mm}K_k^1 \hat{x}_{k|k-1}^1)
\Big] \notag\\
& + \mathbb{E}\!\left[
x_0^\top P_0^1 \hat{x}_{0|-1}^1
+ x_0^\top \Phi_0^1 (\hat{x}_{0|-1}^2-\hat{x}_{0|-1}^1)
\right]+ \operatorname{tr}\!\left[\Sigma_{N+1|N}^1 P_{N+1}^1\right] \notag\\
& + \sum_{k=0}^N \operatorname{tr}\!\Big[
\Sigma_{k|k-1}^1\bigl(Q_1 + A^\top P_{k+1}^1 A G_{k|k-1}^{1} C_1\bigr)
\notag\\
&+ (\Sigma_{k|k-1}^1 - \Sigma_{k|k-1}^2)\Big(
-(A + B_2 K_k^2)^\top \Phi_{k+1}^1 A G_{k|k-1}^{1} C_1
\notag\\
&+ (K_k^2)^\top B_2^\top \Phi_{k+1}^1 A G_{k|k-1}^{1} C_1
\Big)\notag\\
&+ \Sigma_{k|k-1}^2\bigl(
A^\top \Phi_{k+1}^1 A G C
- A^\top \Phi_{k+1}^1 A G_{k|k-1}^{1} C_1
\bigr)
\Big]. \label{gjj}
\end{align}
To minimize \eqref{gjj}, it follows that  $\hat{u}_k^*=K_k^1\hat{x}_{k|k-1}^1, \forall k \in \{0, \dots, N\}$, where $K_k^1$ satisfies \eqref{eq:Ric1_clean}. Then the optimal cost function $J^{1}(\{\hat{u}_k^{*}\}_{k=0}^N,\{\tilde{u}_k^{2,*}\}_{k=0}^N)$ in \eqref{oos1} can be obtained. 
By a similar derivation, the optimal actions $\tilde{u}_k^{2,*}$ and corresponding optimal cost $J^{2}(\{\hat{u}_k^{*}\}_{k=0}^N,\{\tilde{u}_k^{2,*}\}_{k=0}^N)$ can be derived.
% gain $K_k^2$ and the corresponding optimal cost $J^{2}(\{\hat{u}_k^{*}\}_{k=0}^N,\{\tilde{u}_k^{2,*}\}_{k=0}^N)$.
\end{proof}

\subsection{Proof of Theorem \ref{thm2}}\label{app2}

\begin{proof} 
First of all, based on the definition of $\Delta_k$ in \eqref{delta}, we denote
\(
S_{k+1}=-A^\top \Delta P_{k+1}^1 A,
\)
together with the cyclic property of the trace yields
\begin{small}
\begin{align}
&\operatorname{Tr}(\Sigma_{k|k-1}^1 S_{k+1} G_{k|k-1}^{1} C_1)
-\operatorname{Tr}(\Sigma_{k|k-1}^2 S_{k+1} G_{k|k-1}^2 C) \notag\\
&\qquad
= \operatorname{Tr}(\Delta P_{k+1}^1 A \Delta\Sigma_{k|k} A^\top)
-\operatorname{Tr}(\Delta P_{k+1}^1 A \Delta\Sigma_{k|k-1} A^\top).
\end{align}
\end{small}
Moreover, by Lemma~\ref{kf_lem}, it follows that
\begin{align}
&G_{k|k-1}^{1} C_1 \Sigma_{k|k-1}^1
-
G_{k|k-1}^2 C \Sigma_{k|k-1}^2=
\Delta\Sigma_{k|k-1}-\Delta\Sigma_{k|k}.
\end{align}
Substituting this identity into \eqref{jj3} yields
% \begin{align}
% &\Delta_k= \operatorname{Tr}\left[Q_1 \Delta\Sigma_{k|k-1}
% + \Delta P_{k+1}^1 \Delta\Sigma_{k+1|k}\notag\\
% &-\Delta P_k^1 \Delta\Sigma_{k|k-1} - \Delta P_{k+1}^1 \Delta_k^u\right].
% \end{align}
\begin{align}\label{delta2}
\Delta_k
&= \operatorname{Tr}\Big[
Q_1 \Delta \Sigma_{k|k-1}
+ \Delta P_{k+1}^1 \Delta \Sigma_{k+1|k} \notag\\
&\qquad
+ \Delta P_k^1 \Delta \Sigma_{k|k-1}
- \Delta P_{k+1}^1 \Delta_k^u
\Big],
\end{align}
by summing the above equality over the horizon $k=0,\ldots,N$ then gives the decomposition in \eqref{delta}.

Now, to complete the proof, we need to represent that each term of \eqref{eq:gap_decomp_final} is nonnegative.
% As a result, it is desirable to be such that each term is non-negative to prove \eqref{eq:gap_decomp_final} holds.

First, due to $\Phi_0^1 \succeq 0$ and $\Delta\Sigma_{k|k-1}\succeq 0$, hence the quadratic term 
\(
(\hat{x}_{0|-1}^2-\hat{x}_{0|-1}^1)^\top \Phi_0^1 (\hat{x}_{0|-1}^2-\hat{x}_{0|-1}^1)
\)
is nonnegative. Likewise, since $Q_1 \succeq 0$ and $\Sigma_{k|k-1}^1-\Sigma_{k|k-1}^2=\Delta\Sigma_{k|k-1}\succeq 0$, we have
\(
\operatorname{Tr}\bigl(Q_1(\Sigma_{k|k-1}^1-\Sigma_{k|k-1}^2)\bigr)\ge 0.
\)

Second, consider the weighting matrix difference $\Delta P_k^1$. From \eqref{eq:Ric1_clean}, it satisfies the backward recursion
\begin{align}
\Delta P_k^1
&= (A+BK_k^1)^\top \Delta P_{k+1}^1 (A+BK_k^1) + Q_1 \notag\\
&\quad + (K_k^1)^\top \Gamma_1 K_k^1
+ (A+BK_k^1)^\top \Phi_{k+1}^1 (A+BK_k^1).
\end{align}
We next show, by backward induction, that
\(
\Delta P_k^1 \succeq 0
\)
for all \(k=0,\ldots,N\).
At the terminal step, since $P_{N+1}^1=I$ and $\Phi_{N+1}^1=0$, we have
\(
\Delta P_{N+1}^1=I\succ 0.
\)
Together with $Q_1\succ 0$ and $\Gamma_1\succ 0$, this implies $\Delta P_N^1\succeq 0$.
Repeating the same argument backward for \(k=N-1,\dots,0\) yields
\(
\Delta P_k^1\succeq 0,\quad k=0,\ldots,N.
\)
Consequently, we have
\(
\operatorname{Tr} \bigl(\Delta P_{k+1}^1(\Sigma_{k+1|k}^1-\Sigma_{k+1|k}^2)\bigr)\ge 0
\)
and
\(
\text {Tr}\bigl(\Delta P_k^1(\Sigma_{k|k-1}^1-\Sigma_{k|k-1}^2)\bigr)\ge 0.
\)

Finally, for the residual term, since $\Sigma_{k|k-1}^1-\Sigma_{k|k-1}^2\succeq 0$ and $B_2K_k^2 \neq 0$, we have
\begin{align}
\Delta_k^u
=
B_2K_k^2\bigl(\Sigma_{k|k-1}^2-\Sigma_{k|k-1}^1\bigr)(K_k^2)^\top B_2^\top
\preceq 0.
\end{align}
Because of $\Delta P_{k+1}^1\succeq 0$, it follows that
\(
-\operatorname{Tr} \bigl(\Delta P_{k+1}^1\Delta_k^u\bigr)\ge 0.
\)

Based on the above results, each term in the equation \eqref{eq:gap_decomp_final} is a nonnegative value. We can conclude that  \eqref{eq:gap_decomp_final} holds. This completes the proof.
% Combining the above results and invoking the trace property, each term of \eqref{eq:gap_decomp_final} is nonnegative. 
\end{proof}

\end{document}